\documentclass[11pt]{article}
\usepackage{amsmath, amssymb}
\usepackage[style=alphabetic,sorting=nyt,sortcites=true,maxbibnames=10,autopunct=true,babel=hyphen,hyperref=true,abbreviate=false,backref=true,backend=biber]{biblatex}
\usepackage[new]{old-arrows}
\usepackage{stackengine}
\stackMath
\usepackage{mathtools}

\usepackage{amsfonts}
\usepackage{subfigure}
\usepackage{titlesec}

\usepackage{pst-node}
\usepackage{mathrsfs}
\usepackage[left=3cm, right=3cm, top=3cm]{geometry}
\usepackage[arrow,matrix,curve,cmtip,ps]{xy}
\usepackage{graphicx}
\usepackage{comment}

\usepackage[utf8]{inputenc}
\usepackage{mathtools}
\usepackage{tikz-cd}
\usepackage{float}
\usepackage{amsthm}
\usepackage{enumitem}
\usepackage[pdftex]{pict2e}
\usepackage{hyperref}
\usepackage{ytableau}

\makeatletter
\let\@wraptoccontribs\wraptoccontribs
\makeatother

\hypersetup{
    colorlinks=true,
    linkcolor=blue,
    filecolor=magenta,      
    urlcolor=cyan,
    pdftitle={Overleaf Example},
    citecolor=blue,
    pdfpagemode=FullScreen,
    }

\allowdisplaybreaks

\newtheorem{thm}{Theorem}[section] 

\newtheorem{mainthm}{Theorem}
\newtheorem*{mainthm'}{Main Theorem}

\newtheorem{defn}[thm]{Definition} 

\newtheorem{lem}[thm]{Lemma}
\newtheorem{prop}[thm]{Proposition}
\newtheorem{cor}[thm]{Corollary}
\newtheorem{ex}[thm]{Example}

\newtheorem{rmk}[thm]{Remark}

\newcommand{\Z}{\mathbb{Z}}
\newcommand{\N}{\mathbb{N}}
\newcommand{\R}{\mathbb{R}}
\newcommand{\C}{\mathbb{C}}

\newcommand{\cV}{\mathcal{V}}

\newcommand{\ind}{\operatorname{Ind}}
\newcommand{\Res}{\mathrm{Res}}
\newcommand{\alt}{\mathrm{alt}}

\newcommand{\cH}{\mathcal{H}}
\newcommand{\cA}{\mathcal{A}}
\newcommand{\cB}{\mathcal{B}}
\newcommand{\cE}{\mathcal{E}}
\newcommand{\cD}{\mathcal{D}}
\newcommand{\cC}{\mathcal{C}}
\newcommand{\cK}{\mathcal{K}}
\newcommand{\cP}{\mathcal{P}}
\newcommand{\cR}{\mathcal{R}}
\newcommand{\cI}{\mathcal{I}}
\newcommand{\cM}{\mathcal{M}}

\newcommand{\cF}{\mathcal{F}}
\newcommand{\cZ}{\mathcal{Z}}

\newcommand{\ms}{\mathrm{ms}}
\newcommand{\e}{\mathbf{e}}
\newcommand{\std}{\mathrm{std}}

\newcommand{\norm}[1]{\left\lVert #1 \right\rVert}
\newcommand{\inprod}[1]{\left\langle #1 \right\rangle}

\newcommand{\Hmb}{\mathrm{H}_\mathrm{b}}

\newcommand{\diam}{\mathop{\mathrm{diam}}}

\newcommand{\id}{{\mathrm{id}}}
\newcommand{\Hcb}{\mathrm{H}_\mathrm{cb}}

\newcommand{\ev}{\mathrm{ev}}
\newcommand{\bC}{\operatorname{C_\mathrm{b}}}

\newcommand{\Sym}{\mathrm{Sym}}

\newcommand{\bd}{\partial}

\renewcommand{\bd}[1]{\partial #1}
\newcommand{\aut}{\mathrm{Aut}}

\newcommand{\stab}{\mathrm{Stab}}

\newcommand{\sym}{\mathrm{Sym}}
\newcommand{\sgn}{\mathrm{sgn}}
\newcommand{\Ind}{\mathrm{Ind}}
\newcommand{\wt}{\widetilde}
\newcommand{\perm}{\mathrm{perm}}

\title{Bounded Cohomology and Unitary Representations of Automorphism Groups of Regular Trees}
\author{Cunyuan Zhao\footnote{\'Ecole Polytechnique F\'ed\'erale de Lausanne, Lausanne, Switzerland. Email Address: cunyuan.zhao@epfl.ch\\ ORCID: 0009-0008-4773-2053.}}
\date{\vspace{-5.5ex}}

\begin{document}

\maketitle

\begin{abstract}
   Let $G=\aut(T)$ be the group of automorphisms of a regular tree $T$. We compute the continuous bounded cohomology $\Hcb^*(G,\cH_\pi)$ of $G$ in all positive degrees, with coefficients arising from any irreducible continuous unitary representations $(\pi, \cH_\pi)$ of $G$. To the author's knowledge, this seems to be the first instance where the continuous bounded cohomology is determined in all positive degrees with coefficients in any irreducible continuous unitary representations without being zero in all cases.
\medbreak
\end{abstract}

\maketitle


\section{Introduction}

In this paper, we consider the automorphism group $\aut(T)$ of a regular tree $T$ equipped with the topology of point-wise convergence. In analogy to the work of Burger and Monod on $\mathrm{SL}_2(\R)$ (\cite{BM02a}), we determine the continuous bounded cohomology $\Hcb^*(\aut(T), \cH_\pi)$ of $\aut(T)$ with coefficients in any irreducible continuous unitary representations $(\pi, \cH_\pi)$ of $G$. Whereas \cite{BM02a} determined $\Hcb^n(\mathrm{SL}_2(\R),-)$ for $n\leq 2$, we are able to compute $\Hcb^n(\aut(T),-)$ for all $n\geq 1$. 

Our result contrasts with the usual $n$-th continuous cohomology $\mathrm{H}_\mathrm{c}^n(\aut(T), \cH_\pi)$, which is known to vanish for all $n\geq 2$ (see \cite{BW00}, Lemma 1.12 in Chapter X) and is non-trivial for exactly one irreducible continuous unitary representation when $n=1$ (\cite{Neb12}). 

Bounded cohomology of groups of non-trivial coefficients has witnessed a wide range of applications. Indeed, a fundamental result of Johnson in \cite{Joh72} states that a group is amenable if and only if it has vanishing bounded cohomology to all positive degrees with coefficients in any dual Banach $G$-modules. Bounded cohomology with coefficients in unitary representations of groups arises in the study of orbit equivalence rigidity (Monod and Shalom, \cite{MS06}), and cocycle superrigidity of groups acting on negatively curved spaces (Monod and Shalom, \cite{MS04}). 

Recent progress in bounded cohomology that goes beyond the trivial coefficient includes the vanishing of the bounded cohomology of lamplighters and the Thompson group $F$ in all positive degrees with coefficients in any separable dual Banach $G$-module (Monod, \cite{Mon22}). This was further extended to groups with commuting cyclic conjugates by Campagnolo et al. (\cite{CFFLM24}). 

We now state a rough version of our main result, which will be made precise in Theorem \ref{thm: A}, \ref{thm: B} and \ref{thm: C} below. Let $T$ be a $(q+1)$-regular tree $(q\geq 2)$. Denote by $\widehat{G}$ the collection of all equivalence classes of irreducible continuous unitary representations of $G:=\aut(T)$.

\begin{mainthm'}
Let $n$ be a positive integer. Then the $n^{\textrm{\emph{th}}}$ continuous bounded cohomology $\Hcb^n(G,\cH_\pi)$ vanishes whenever $n\ne 2$ for any $(\pi,\cH_\pi)\in \widehat{G}$.

Moreover, there is a countably infinite family $\cC_G\subset \widehat{G}$ such that 
$\Hcb^2(G,\cH_\pi)$ is non-trivial and 1-dimensional
for all $(\pi,\cH_\pi)\in \cC_G$, and 
$\Hcb^2(G,\cH_\pi)=0$
for any $(\pi,\cH_\pi)\in \widehat{G}\setminus \cC_G$. 
\end{mainthm'}

To the author's knowledge, our computation of $\Hcb^*(\aut(T), \cH_\pi)$ seems to be the first instance where the continuous bounded cohomology is determined in every positive degree with coefficients in any irreducible unitary representations without being zero in all cases. 
\bigbreak

To be precise about $\cC_{G}$ and the corresponding bounded cohomologies, we introduce a few notations and definitions (see Section \ref{sec: cuspidal_intro} for more details). 

From now on, all unitary representations of $G$ are assumed to be continuous (with respect to the weak or, equivalently, strong operator topology). For a subtree $S$ in $T$, denote by $G(S)$ (resp. $\wt{G}(S)$) its point-wise (resp. set-wise) stabilizer in $G$. A non-empty subtree $S$ in $T_q$ is \textbf{complete} if either it is a vertex, or if any vertex in $S$ has degree $1$ or $q+1$. If $(\pi, \cH_\pi)\in \widehat{G}$, then one can show that there exists some finite complete subtree $S$ such that the space of $G(S)$-invariant vectors $\cH_\pi^{G(S)}$ is non-trivial (see Section III.1 in \cite{F-TN91}). We denote by $\cM_\pi$ the set of minimal finite complete subtree $S$ such that $\cH_\pi^{G(S)}\ne 0$. It is known that $G$ acts transitively on $\cM_\pi$ (see Lemma 23 in \cite{Ama03}), and the shape of finite complete subtrees in $\cM_\pi$ plays a fundamental role in the classification of irreducible unitary representations of $G$ by Ol'shanski (\cite{Ol77}, see also \cite{F-TN91}).

\begin{defn}\label{def_int: heads}
\emph{Let $S$ be a finite complete subtree with $\diam(S)\geq2$. A \textbf{head} of $S$ is the minimal subtree in $S$ that contains the vertices in $S\setminus S'$, where $S'$ is a maximal complete proper subtree of $S$. We say that $S$ is $k$-\textbf{headed} if it has $k$ heads.}
\end{defn}

We now give names to special classes of finite complete subtrees:

\begin{defn}\label{def_int: centipede}
\emph{A $2$-headed finite complete subtree $S$ of diameter $k\geq 3$ is called an $k$-\textbf{centipede}. Any finite complete subtree of diameter $2$ is called a \textbf{spider}.}
\end{defn}

\begin{figure}[!htb]
   \centering
    \setlength{\unitlength}{0.5cm}
\begin{picture}(10,3)
\put(5,2){\line(1,0){2}}
\put(5,2){\line(1,-1){0.8}}
\put(5,2){\line(-1,-1){0.8}}
\put(5,2){\line(-1,0){2}}
\linethickness{0.7mm}
\put(7,2){\color{red}\line(1,0){1.4}}
\put(7,2){\color{red}\line(1,-1){0.8}}
\put(7,2){\color{red}\line(-1,-1){0.8}}
\put(3,2){\color{red}\line(-1,0){1.4}}
\put(3,2){\color{red}\line(1,-1){0.8}}
\put(3,2){\color{red}\line(-1,-1){0.8}}
\end{picture}
\begin{picture}(10,3)
\put(5,2){\line(1,0){1}}
\put(5,2){\line(1,-1){0.7}}
\put(5,2){\line(-1,-1){0.7}}
\put(5,2){\line(0,-1){1}}
\put(5,2){\line(-1,0){1}}
\put(5,2){\line(1,1){0.7}}
\put(5,2){\line(-1,1){0.7}}
\put(5,2){\line(0,1){1}}
\end{picture}
    \caption{A $4$-centipede in a $4$-regular tree (left) and a spider in an $8$-regular tree (right). The two heads of the $4$-centipede are indicated by thick red segments.}
    \label{fig:centipede0}
\end{figure}
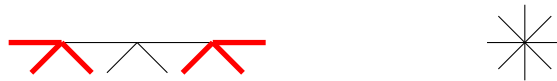

We now state Theorem \ref{thm: A}:
\begin{mainthm}\label{thm: A}
Let $(\pi,\cH_\pi)\in \widehat{G}$ be a unitary irreducible representation of $G$. If $\cM_\pi$ does not consist of centipedes or spiders, then $\Hcb^n(G, \cH_\pi)=0$ for all $n\geq 1$.
\end{mainthm}

This contrasts with Theorem 1.1 in \cite{BM02a}, which states that $\Hcb^2(\mathrm{SL}_2(\R), \cH_\pi)\ne 0$ if and only if $\pi$ is a unitary spherical representation. Here, we recall that unitary spherical representations of $\aut(T)$ correspond to the cases where $\cM_\pi$ consists of single vertices. See Section \ref{sec: sph_intro} for more details. 

We introduce a bit more background on unitary representations of $G$ to state our results on centipedes and spiders. Let $S$ be a $k$-headed finite complete subtree of diameter at least 2, and let $S_1,...,S_k$ be its maximal proper complete subtrees. Set $A_j:=p_S(G(S_j))$. We say $(\omega, V)$ of the finite group $\wt{G}(S)/G(S)\simeq \aut(S)$ is \textbf{non-degenerate} if there are no non-trivial $A_j$-invariant vectors for all $j=1,...,k$. As shown in \cite{Ol77}, if $S\in \cM_\pi$ has diameter at least 2, then $(\pi, \cH_\pi)$ must be induced by a unitary representation of $\wt{G}(S)$ that factors through some non-degenerate irreducible unitary representation $(\omega, V)$ of the finite group $\wt{G}(S)/G(S)\simeq \aut(S)$. In this case, we write $\pi:=\ind_S\, \omega$.

Now suppose that $S$ is a centipede or a spider. Fix $x\in S\setminus S_1$ and $y\in S\setminus S_2$ and denote by $Q=Q(x,y)$ the point-wise stabilizer of $\{x,y\}$ and by $\wt{Q}=\wt{Q}(x,y)$ the set-wise stabilizer of $\{x,y\}$. Note that $Q$ and $\wt{Q}$ does not depend on the choice of $x$ and $y$ in $S\setminus S_1$ and $S\setminus S_2$.  

\begin{mainthm}\label{thm: B}
Let $S$ be a centipede of a spider and let $(\omega, V)$ be a non-degenerate irreducible unitary representation of $\aut(S)$. Then $\dim V^{Q}/V^{\wt{Q}}\leq 1$ and
$$
\Hcb^n(G,\ind_S\,\omega)\simeq\begin{cases}
    V^{Q}/\, V^{\wt{Q}} & \textrm{if }\, n=2\\
    0 & \textrm{if }\, n=1 \textrm{ or }  n\geq 3
    \end{cases}.
$$
\end{mainthm}

Theorem \ref{thm: A} and \ref{thm: B} provide an easily-verifiable criterion to determine whether $\Hcb^2(G,\cH_\pi)$ is nontrivial for a given $\pi\in \widehat{G}$. In fact, we were able to obtain a complete and explicit list of inequivalent representations $\pi\in \widehat{G}$ such that $\Hcb^2(G, \cH_\pi)\ne 0$ (see Section \ref{sec: C}). In particular, one has the following uniqueness result:

\begin{mainthm}\label{thm: C}
Let $S$ be a centipede or a spider of diameter $k\geq 2$. Then up to isomorphism, there is a unique non-degenerate irreducible unitary representation $(\omega_k, V_k)$ of $\aut(S)$ such that 
$\Hcb^2(G, \pi_k)\ne 0$, where $\pi_k :=\ind_S \, \omega_k$.
\end{mainthm}

In other words, the family $\cC_G$ in the Main Theorem is given by $\{\pi_k:k\in \N_{\geq2}\}$. If $S$ is a spider (i.e. if $k=2$), then $\aut(S)$ is the symmetric group $\Sym(q+1)$ and $\omega_2$ is the second exterior product of the standard representation of $\Sym(q+1)$. For $k\geq 3$, $\omega_k$'s are constructed more or less analogously to $\omega_2$. See Section \ref{sec: C} for more details.

\bigbreak
Previously, Monod and Shalom constructed in \cite{MS04} (see also \cite{MS03}) nontrivial bounded cocycles $c_{\ms}^k$ of $G$ with coefficients in a family of unitary representations $\{(\lambda_k, \cH_k)\}_{k\in\N_{\geq2}}$ of $G$. Roughly speaking, $\cH_k$ is some submodule of the $\ell^2$-space of $k$-consecutive edges in $T$ on which $G$ acts by left translations (see Section \ref{sec: D} for more details). It turns out that the cocycles $c_{\ms}^k$ already "contain" all information of $\Hcb^2(G, \pi_k)$, where $\pi_k$'s are the representations mentioned in Theorem \ref{thm: C}:

\begin{mainthm}\label{thm: D}
Let $k\geq 2$. Then the unique representation $(\pi_k,\cH_{\pi_k})\in \widehat{G}$ mentioned in Theorem \ref{thm: C} is a subrepresentation of $(\lambda_k,\cH_k)$.

Moreover, if $P:\cH_k\to\cH_{\pi_k}$ is the orthogonal projection and $P_*: \Hcb^2(G,\lambda_k)\to\Hcb^2(G,\pi_k)$ is the induced map,
then $\Hcb^2(G,\pi_k)\simeq \C\cdot P_*[c_\ms^k]$.
\end{mainthm}
\bigbreak
We mention here an alternative way to see the existence of at least one irreducible unitary representation $(\pi, \cH_\pi)$ such that $\Hcb^2(\aut(T),\cH_\pi)\ne 0$: The same argument in the introduction of \cite{BM02a} and in the proof of Proposition C in \cite{Mon24} works. To illustrate this and give further context, we repeat this argument in the next paragraph.

Let $G=\aut(T)$, and let $\Gamma\leq G$ be a non-abelian free group that is a cocompact lattice in $G$. Then by Corollary 11.1.5 in \cite{Mon01}, there is an injection
$$
L^2i: \Hmb^2(\Gamma)\longhookrightarrow \Hcb^2(G, L^2(G/\Gamma)),
$$
where $L^2(G/\Gamma)$ is the Banach $G$-module on which $G$ acts by left translations. The $G$-module $L^2(G/\Gamma)$ splits as a direct sum $\bigoplus_{\pi\in\widehat{G}} \cH_\pi^{\oplus n_\pi}$ of continuous irreducible unitary representations of $G$ (see Proposition 9.2.2 in \cite{DE14}) with finite multiplicities $n_\pi$. Due to the double ergodicity of the action $G$ on the boundary $\bd T$ of $T$ with coefficients (11.2.2 in \cite{Mon01}), there is no coboundary in $L_\alt^\infty((\bd T)^3, \cH_\pi)$ (see Corollary \ref{cor: reso_bd} in Section 4 for more background). Thus, the vector space $\Hcb^2(G, L^2(G/\Gamma))$ can be viewed as a subspace of the product $\prod_{\pi\in\widehat{G}}\Hcb^2(G, \cH_\pi)^{\oplus n_\pi}$. Since $\Hmb^2(\Gamma)$ is infinite-dimensional (\cite{Bro81}), there must be at least one irreducible unitary representation $(\pi, \cH_\pi)$ of $G$ such that $\Hcb^2(G, \cH_\pi)\ne 0$. 

Thus, our results described all irreducible unitary representations of $G$ that could possibly contribute to the infinite-dimensional space $\Hcb^2(G, L^2(G/\Gamma))$: all are children of centipedes and spiders.


\medbreak

\noindent\textbf{Outline of the paper.} Section \ref{sec: prelim} sets up the notation and conventions. Section \ref{sec: uni_rep} is a survey of the classification of irreducible unitary representations of $\aut(T)$ by Ol'shanski (\cite{Ol77}) and Figa-Talamanca and Nebbia (\cite{F-TN91}). Section \ref{sec: coh} collects useful facts in continuous bounded cohomology. Section \ref{sec: ae} deals with the technical issue of choosing desired representatives in the relevant $L^\infty$ spaces. Section \ref{sec: pf} is devoted to the proof of Theorem \ref{thm: A} and \ref{thm: B}. Theorem \ref{thm: C} is addressed in Section \ref{sec: C}. Finally, Section \ref{sec: D} discusses the Monod-Shalom cocycle and proves Theorem \ref{thm: D}.

\medbreak

\noindent\textbf{Acknowledgement.} This is a part of the author's PhD project supervised by Nicolas Monod at the Ecole polytechnique f\'ed\'erale de Lausanne (EPFL) in Switzerland. I would like to express my deepest gratitude to Nicolas for suggesting the project and for all the enlightening discussions and advice.

\section{Notations and Preliminaries on Trees}\label{sec: prelim}

In this short section, we set up the notations and recall necessary notions about trees and their boundaries. 

A graph $X$ is a pair $(\cV_X,\cE_X)$ where $\cE_X$ the set of \textbf{unoriented edges}, which is a collection of two-element subsets of the \textbf{vertex set} $\cV_X$. We denote by $\wt{\cE_X}$ the set of \textbf{oriented edges} of $X$. It consists of ordered pairs $(x,y)\in \cV_X\times \cV_X$ such that $\{x,y\}\in \mathcal{E}_X$. Write $e_{xy} := (x,y)$. For $e=e_{xy}=(x,y)\in \wt{\mathcal{E}_X}$, write $\overline{e}:=(y,x)=e_{yx}$. For a vertex $x\in \cV_X$, we sometimes abusively write $x\in X$ when there is no confusion. We say that two vertices $x,y\in X$ are \textbf{adjacent} if $\{x,y\}\in \cE_X$, and in this case we write $x\sim y$. The number of adjacent vertex of a given vertex $x$ is the \textbf{degree} of $x$.

A \textbf{finite chain} from $x_0$ to $x_n$ in a graph $X$ is a finite sequence $x_0,...,x_n$ such that $x_j\sim x_{j+1}$ for all $j=0,...,n-1$ and $x_j\ne x_{j+2}$ for all $j=0,...,n-2$. A finite chain $x_0,...,x_n$ is a \textbf{loop} if $x_0=x_n$. A \textbf{tree} $T$ is a graph if there is no loop, and a $(q+1)$-\textbf{regular tree}, denoted by $T_q$, is a tree such that every vertex has degree $q+1$. 

Let $T$ be a tree. For $I=\N\cup \{0\}$ or $\Z$, a sequence $(x_j)_{j\in I}$ such that $x_j\sim x_{j+1}$ and $x_{j}\ne x_{j+2}$ for all $j\in I$ is called an \textbf{infinite chain} (resp. doubly infinite) if $I=\N\cup\{0\}$ (resp. $I=\Z$). A doubly infinite chain is also called a \textbf{geodesic}. The \textbf{boundary at infinity} $\bd T$ of $T$ consists of equivalence classes of infinite chains $(x_j)_{j\in \N\cup\{0\}}$, where two infinite chains are equivalent if their intersection is infinite. For each vertex $x\in T$, the unique infinite chain starting at $x$ that represents a boundary point $\gamma\in \bd T$ is denoted by $[x,\gamma)$, or alternatively, by $(\gamma(x,t))_{t\in \N\cup\{0\}}$ when necessary. For $\gamma,\eta\in \bd T$, the unique geodesic joining $\gamma$ and $\eta$ is denoted by $L(\gamma,\eta)$.

The set $\overline{T}:=\cV_T\cup\bd T$ carries a compact topology such that $\cV_T$ is dense in $\overline{T}$. A basic open set of the subspace topology of $\bd T\subset \overline{T}$ is of the form $U(x,y)$, which consists of all boundary points represented by infinite chains that start from the vertex $x$ and pass through the vertex $y$. 

For any vertex $x\in T$, denote by $\mu_x$ the unique Borel probability measure on $\bd T$ such that
$$\mu_x(U(x,y))= \frac{1}{q+1}\cdot\left(\frac{1}{q}\right)^{d_T(x,y)-1},$$
where $d_T$ is the combinatorial distance on $T$. For any $x,y\in T$, the measures $\mu_x$ and $\mu_y$ are absolutely continuous with respect to each other. The Radon-Nikodym derivatives are given by 
$$\frac{d\mu_y}{d\mu_x}(\gamma) = q^{B_\gamma(x,y)}:=P(x,y, \gamma),$$
where 
$$B_\gamma(x,y):=\lim_{t\to\infty} d_T(\gamma(y,t), x)-d_T(\gamma(x,t),x)$$
is the \textbf{Busemann function} between $x$ and $y$ with respect to $\gamma$. 

Fix a vertex $z\in \cV_T$ and let $x,y\in \overline{T}$. Let $(x_n)$, $(y_n)$ be sequences in $\cV_T$ that converge to $x$ and $y$ respectively. Consider the limit
$$
(x,y)_z := \lim_{n\to \infty} \frac{1}{2}(d_T(x_n,z)+d_T(y_n,z)-d_T(x_n,y_n))
$$
Then the limit converges and is independent of the choice of $(x_n)$ and $(y_n)$. It is called the \textbf{Gromov product} of $x$ and $y$ at $z$. If $\gamma\in \bd T$, we set $(\gamma,\gamma)_x=+\infty$ for any vertex $x\in T$. For a triple $x =(x_0,x_1,x_2)\in (\overline{T})^3$ of pairwise distinct points, the unique vertex $m_x$ that satisfies
$$
(x_0,x_1)_{m_x} = (x_0,x_2)_{m_x} = (x_1,x_2)_{m_x}=0
$$
is called the \textbf{median} of $x_0,x_1$ and $x_2$. 

Now let $q\geq 2$ and let $T=T_q$ be a $(q+1)$-regular tree. Let $\aut(T)$ be the automorphism group of $T$ equipped with the topology of point-wise convergence. Throughout the paper, whenever $G=\aut(T)$ and $S$ is a subtree of $T$, we denote by $G(S)$ the subgroup of point-wise stabilizers of $S$, and $\wt{G}(S)$ the subgroup of set-wise stabilizers of $S$. The quotient $Q(S):=\wt{G}(S)/G(S)$ can be identified with the automorphism group $\aut(S)$ of $S$, and we denote by 
$$p_S: \wt{G}(S)\longrightarrow Q(S)$$
the canonical projection. 

Finally, for any set $X$, we denote by $D_n(X)$ the subset of $X^n$ consisting of tuples $(x_1,...,x_n)$ such that $x_i\ne x_j$ whenever $i\ne j$. Observe that $D_n(\bd T)$ is a conull subset of $(\bd T)^n$ for all $n\geq 1$. 


\section{Unitary Representations of $\aut(T)$}\label{sec: uni_rep}

Let $T=T_q$ be a ${q+1}$-regular tree ($q\geq 2$). In this section, we recall the classification of irreducible unitary representations of $\aut(T)$ given in the work of Ol'shanski (\cite{Ol77}) and Figa-Talamanca and Nebbia (\cite{F-TN91}). Some of the notations we adapt come from \cite{Ama03}.

\begin{defn}\label{def: finite_complete}
\emph{
We say that a non-empty subtree $S$ in $T_q$ is \textbf{complete} if either it is a vertex, or any vertex in $S$ has degree $1$ or $q+1$. 
}
\end{defn}

In this paper, all unitary representations are assumed to be continuous:

\begin{defn}
\emph{A \textbf{unitary representation} $(\pi, \cH_\pi)$ of a topological group $G$ is a homomorphism $\pi$ from $G$ to the group of unitary operators of a Hilbert space $\cH_\pi$ such that for all $u,v\in \cH_\pi$, the map $g\mapsto \inprod{\pi(g)u,v}_{\cH_\pi}$ is continuous.}
\end{defn}

Now let $G=\aut(T)$. Without much difficulty, one can show that if $(\pi, \cH_\pi)$ is irreducible, then there exists some finite complete subtree $S$ such that the space of $G(S)$-invariant vectors $\cH_\pi^{G(S)}$ is non-trivial (see, for example, Section III.1 in \cite{F-TN91}). We denote by $\cM_\pi$ the set of minimal finite complete subtree $S$ such that $\cH_\pi^{G(S)}\ne 0$. It is known that $G$ acts transitively on $\cM_\pi$ (see, for instance, Lemma 23 in \cite{Ama03}), and so $\cM_\pi$ could be identified with the quotient $G/\wt{G}(S)$ for any $S\in \cM_\pi$. 

The discussion in the last paragraph suggests that any irreducible unitary representation of $G$ belongs to one of the three types in Definition \ref{def: classifi} below:

\begin{defn}\label{def: classifi}
\emph{
We say that an irreducible representation $(\pi, \cH_\pi)$ of $G=\aut(T)$ is
\begin{enumerate}[nosep]
\smallbreak
    \item a \textbf{unitary spherical representation} if $\cM_\pi$ consists of vertices of $T$;
    \item a \textbf{unitary special representation} if $\cM_\pi$ consists of edges of $T$;
    \item a \textbf{unitary cuspidal representation} if $\cM_\pi$ consists of finite complete subtrees of diameter at least 2.
\end{enumerate}
An element $S_0\in \cM_\pi$ is called a \textbf{minimal tree} of $\pi$.
}
\end{defn}

The irreducible unitary representations of $\aut(T)$ was completely classified in the work of Ol'shanski (\cite{Ol77}). A detailed account for the classification of unitary spherical and special representations for all closed subgroups of $\aut(T)$ acting transitively on $\cV_T$ and $\bd T$ can be found in \cite{F-TN91}. We remark here that although our method only applies to the full automorphism group of $T$, the classification in \cite{Ol77} and \cite{F-TN91} was generalized to a class of subgroups with the "independence property" by Amann in \cite{Ama03}.

\subsection{Unitary Spherical Representations and Smooth Vectors}\label{sec: sph_intro}

\subsubsection{Unitary Spherical Representations}
We follow \cite{F-TN91}, Chapter II to recall the classification of unitary spherical representations of $G=\aut(T)$.  

Unitary spherical representations of $G$ are characterized by positive definite spherical functions. From now on, fix a vertex $o\in T$ and let $K$ be the stabilizer subgroup of $o$. 

\begin{defn}\label{def: pos_def}
\emph{A function $f: G\to\C$ is said to be \textbf{positive definite} if for any $n\in \N$, $c_1,...,c_n\in \C$, and $g_1,...,g_n\in G$, one has}
$$\sum_{i,j}c_i\overline{c_j}f(g_ig_j^{-1})\geq 0.$$
\end{defn}

Let $V$ be a vector space. For a function $f: T\to V$ defined on vertices, denote by $Af(x)$ the average of $f$ on the neighbors of $x\in T$.

\begin{defn}\label{def: radial_fcn}
\emph{A function $f: T\to \C$ defined on vertices is \textbf{radial} if $d_T(o,x)=d_T(o,y)$ implies $f(x)=f(y)$. Note that any radial function can be identified with a $K$-bi-invariant function on $G$ under the map $g\mapsto g\cdot o$. A function $\varphi$ is called a \textbf{spherical function} if it is a radial eigenfunction of the averaging operator $A$ such that $\varphi(o)=1$.}
\end{defn}

By Theorem 2.1, Chapter II in \cite{F-TN91}, if $\varphi_z$ is a spherical function that corresponds to the eigenvalue
$\mu(z):= ({q^z+q^{1-z}})/({q+1})$, then $\varphi_z$ is given by the formula
$$\varphi_z(x) = \int_{\bd T} P^z(o,x,\omega)\;d\mu_o(\omega).$$
By Section 5, Chapter II in \cite{F-TN91}, the spherical function $\varphi_z$ is positive definite if and only if $\mu(z) \in[-1,1]$, which holds if and only if $\mathrm{Re}(z)=1/2$, or $0\leq\mathrm{Re}(z)\leq 1$ and $\mathrm{Im}(z)={k\pi}/{\ln (d-1)}$ for some $k\in \Z$. For such $z$, a unitary representation $\pi_z$ of $G$ is defined as follows: Denote by $\cK_z$ the vector space generated by the left translates of $\varphi_z$. If $f_1(g) = \sum_i c_i\varphi_z(g_i^{-1}g)$ and $f_2(g) = \sum_j d_j\varphi_z(g_j^{-1}g)$, then define
$$\inprod{f_1,f_2}_z:= \sum_{i,j} c_i\overline{d_j} \varphi_z(g_i^{-1}g_j)$$
Since $\varphi_z$ is positive definite, the above defines an inner product on $\cK_z$. Let $\cH_z$ be the completion of $\cK_z$ with respect to $\inprod{\,,}_z$. Then $(\pi_z, \cH_z)$ is defined by $\pi_z(h)f(g) = f(h^{-1}g)$ for all $h\in G$, $f\in \cK_z$. 

It is known that every unitary spherical representation is isomorphic to $\pi_z$ for some $z$ with $\mu(z)\in [-1,1]$ (\cite{F-TN91}, Theorem 5.3, Chapter II).  
\begin{defn}\label{def: sph_series}
\emph{The collection of representations $\{\pi_z: \mathrm{Re}(z)=1/2\}$ is called the \textbf{unitary principal series}, and the collection of representations $\{\pi_{s+\frac{k\pi}{\ln q}i}: s\in[0,1], k\in \Z\}$  is called the \textbf{unitary complementary series}.}
\end{defn}

A useful alternative description of $\pi_z$'s was given in Chapter II in \cite{F-TN91} when $\mu(z) \in (-1,1)$, and we briefly outline it here. Let $\cK(\bd T)$ be the vector space of functions on $\bd T$ that are linear combinations of characteristic functions on $U(o,x)$. There is a well-defined intertwining operator $I_z: \cK(\bd T)\longrightarrow \cK(\bd T)$ defined as follows: for $\varphi\in \cK(\bd T)$, the function $I_z\varphi\in \cK(\bd T)$ satisfies
$$\int_{\bd T} P^z(o,x,\gamma)(I_z\varphi)(\gamma)\,d\mu_o(\gamma) = \int_{\bd T}P^{1-z}(o,x,\gamma)\varphi(\gamma)\,d\mu_o(\gamma).$$
For $\varphi, \psi\in \cK(\bd T)$, define
$$(\phi,\psi)_z:=\int_{\bd T} I_z\varphi\cdot\overline{\psi}\;d\mu_o.$$
Using the fact that $\varphi_z$ is positive definite, one can show that $(\,,)_z$ defines an inner product on $\cK(\bd T)$. Denote by $\cH_z'$ the completion of $\cK(\bd T)$ with respect to $(\,,)_z$, and define a unitary representation $(\pi_z',\cH_z')$ by
$$\pi_z'(g)\varphi(\gamma):=P^z(o,go,\gamma)\varphi(g^{-1}\gamma)$$
for all $g\in G$, $\varphi\in \cK(\bd T)$ and $\gamma\in \bd T$. The representation $(\pi_z',\cH_z')$ is unitarily equivalent to $(\pi_z,\cH_z)$. The subspace $\cK_z$ in $\cH_z'$ corresponds to the subspace $\cK(\bd T)$ in $\cH_z'$, and the action of $G$ on $\cK_z$ by left translation corresponds to the action of $G$ on $\cK(\bd T)$ defined by $\pi_z'$. Note that if $\pi_z$ is in the principal series, then $\cH_z'=L^2(\bd T)$.

\subsubsection{An Interlude on Smooth Vectors}

In this subsection, we recall the notion on smooth vectors and representations, which will be relevant to our proof of the spherical cases of Theorem \ref{thm: A}.

Recall that a representation $(\pi, V)$ of $G=\aut(T)$ on a complex inner product space $V$ is \textbf{admissible} if the subspace $V^H$ of vectors fixed by any compact open subgroup $H\leq G$ is finite-dimensional. A vector $v\in V$ is said to be \textbf{smooth} it is fixed by some compact open subgroup. The subspace of smooth vectors is denoted by $V^\infty$. The representation $(\pi, V)$ is \textbf{smooth} if $V=V^\infty$. 
\begin{ex}\label{ex: smooth}
\emph{For $z\in \C$ with $\mu(z)\in (-1,1)$, any vector $\varphi\in \cK(\bd T)$ is smooth in the representation $(\pi_z', \cH_z')$. Indeed, if $x_1,...,x_n\in T$ and $\varphi$ is a linear combination of the characteristic functions $\mathbf{1}_{U(o,x_j)}$, then $\varphi$ is fixed by the compact open subgroup $H=K\cap\bigcap_{j=1}^n\stab_G(x_j)$. Likewise, any vector in $\cK_z$ is smooth in $(\pi_z,\cH_z)$. The representation $(\pi_z', \cK(\bd T))$ (resp. ($\pi_z, \cK_z$)) is a smooth representation.}
\end{ex}

\begin{lem}\label{lem: smooth}
Let $(\pi, V)$ be a smooth unitary representation of $G$ on some complex inner product space $V$. Suppose that $(\bar{\pi}, \overline{V})$ is the admissible unitary representation on the completion $\overline{V}$ of $V$ such that $\bar\pi(g)v = \pi(g)v$ for all $v\in V$, then $V = (\overline{V})^\infty$.
\begin{proof}
It is clear that $V\subset (\overline{V})^\infty$. Conversely, suppose that $v\in \overline{V}$ is fixed by a compact open subgroup $H\leq G$, then $v$ is the limit of a sequence $(v_n)$ in $V$. Consider the projection $\Phi_H: \overline{V}\to \overline{V}^H$ given by
$$\Phi_H(v) := \int_H\bar{\pi}(g)v\; d\mu(g),$$
where $\mu$ is the normalized Haar measure on $K$. Then observe that $\Phi_H(w)\in V^H$ for all $w\in V$. Indeed, since $(\pi, V)$ is smooth, every $w\in V$ is fixed by another compact open subgroup $L$ of $G$. Consider the subgroup $L' := L\cap H$ of $H$. Then the map $K\to V$ defined by $g\mapsto \pi(g)w$ is constant on the left cosets of $L'$. Hence $\Phi_H(w)$ is a finite sum of elements in $V$, so $\Phi_H(w)\in V^H$. To conclude, we have $v = \Phi_H(v)= \lim_{n\to \infty} \Phi_H(v_n)$. The subspace $V^H$ is finite-dimensional in admissibility, and so $(\Phi_H(v_n))$ converges to an element in $V^H$. 
\end{proof}
\end{lem}

Combining Example \ref{ex: smooth} and Lemma \ref{lem: smooth}, we obtain the following corollary:
\begin{cor}\label{cor2: smooth}
Let $z\in \C$ be such that $\mu(z)\in (-1,1)$. Then $\cH_z^\infty = \cK_z$ and $(\cH_z')^\infty=\cK(\bd T)$.
\end{cor}

\subsection{Unitary Special Representations}
Let $G=\aut(T)$. In this subsection, we follow \cite{F-TN91}, Section 2 in Chapter III to recall the description of unitary special representations of $G$. 

Consider the $G$-module $(\lambda, \ell^2(\wt{\mathcal{E}_T},\C))$ of $\ell^2$-functions on oriented edges on which $G$ acts by left translations $\lambda$. Denote by $\cH_\lambda$ the submodule consisting of all $\ell^2$-functions $f: \wt{\mathcal{E}_T}\to \C$ such that for any vertices $x\in T$, one has
$$\sum_{y\sim x} f(e_{xy})=\sum_{y\sim x}f(e_{yx})=0.$$ 
Fix any $e_0\in \wt{\mathcal{E}_T}$ and choose any inversion $g_0\in G$ such that $g_0e_0=\overline{e_0}$. Then the $(\pm 1)$-eigenspaces $\cH_\lambda^\pm$ of $\lambda(g_0)$ in $\cH_\lambda$ are $G$-invariant, and it was shown that they are the only two unitary inequivalent special representations of $G$ (\cite{F-TN91}, 2.6 in Chapter III).  

\subsection{Induced Representations and Unitary Cuspidal Representations}\label{sec: cuspidal_intro}

To describe the classification of unitary Cuspidal representations of $\aut(T)$ by Ol'shanski, we recall the notion of induced representations. Let $G$ be a unimodular locally compact second countable group and let $H$ be a compact open subgroup. Let $(\omega, V)$ be a unitary representation of $H$ on a Hilbert space $V$. For $1\leq p\leq \infty$, denote by $\cI^p V$ the Banach space of functions $f: G\to V$ such that
\begin{enumerate}[noitemsep]
    \item $f(xh)=\omega(h^{-1})f(x)$ for all $x\in G$ and $h\in H$,
    \item $f$ is bounded if $p = \infty$, and one has
    $$\norm{f}_p^p := \int_G \norm{f(g)}^p_V\, d\mu_G = \sum_{x\in \cR}\norm{f(x)}_V^p<+\infty$$
    if $p\ne \infty$,
\end{enumerate}
where $\mu_G$ is the Haar measure on $G$ with $\mu_G(H)=1$, and $\cR$ is a complete set of representatives for $G/H$. We define the $L^p$-\textbf{induced representation} $L^p\ind_H^G \omega$ of $\omega$ to be the left regular representation $\lambda$ of $G$ on $\cI^pV$. 
If $p=2$, then $\cI^2V$ is a Hilbert space with respect to the following inner product:
$$
\inprod{f,g} := \int_G \inprod{f(x),g(x)}_V \,d\mu_G(x)\quad \forall f,g\in \cI^2V.
$$
The representation $L^2 \ind_H^G \omega$ is then unitary, and we abbreviate it as $\ind_H^G \omega$.

We give a useful alternative description of $L^p\ind_H^G \omega$. Let $s: G/H\to G$ be a set-theoretic section. Define an action $\lambda_s$ of $G$ on the Banach space $\ell^p(G/H, V)$ by
$$
(\lambda_s(g))f(xH) := \omega(s(xH)^{-1}gs(g^{-1}xH))f(g^{-1}xH)
$$
for any $f\in \ell^p(G/H,V)$, $x\in \cR$ and $g\in G$. Then $(\lambda_s, \ell^p(G/H,V))$ is isomorphic to $(\lambda, \cI^pV)$ via the linear isometry $\Phi: \cI^pV\to\ell^p(G/H, V)$ defined by $\Phi(f)(xH):=f(s(xH)).$

Now let $G=\aut(T)$ and let $S$ be a finite complete subtree of diameter at least 2. For a unitary representation $(\omega, V)$ of $Q(S)$, denote by $\wt{\omega}:=\omega\circ p_S$ and write
$$\ind_{S}\,\omega :=\ind_{\wt{G}(S)}^G\wt{\omega}.$$ 
If $S_1,...,S_n$ are the maximal proper complete subtree of $S$, then observe that 
$$G(S)\leq G(S_j)\leq \wt{G}(S)$$
for all $j=1,...,n$. Set $A_j:=p_S(G(S_j))$. We say $(\omega, V)$ of the finite group $Q(S)\simeq \aut(S)$ is \textbf{non-degenerate} if there are no non-trivial $A_j$-invariant vectors for all $j=1,...,n$. 

\begin{thm}[\cite{Ol77}, Theorem 1; See also III.3.13 in \cite{F-TN91} and Theorem 1 in \cite{Ama03}]\label{thm: class_cuspidal}
Let $S_0$ be a finite complete subtree of diameter at least $2$. Then the map
$\omega\mapsto \ind_{S_0}\, \omega$
gives a bijection from the set of equivalence classes of non-degenerate irreducible unitary representations of $Q(S_0)$, to the set of equivalence classes of unitary cuspidal representations of $G$ with minimal tree $S_0$. Moreover, if $\pi$ and $\pi'$ are unitary cuspidal representations of $G$ and $\cM_\pi\ne \cM_{\pi'}$, then $\pi$ and $\pi'$ are inequivalent.
\end{thm}


\section{Continuous Bounded Cohomology}\label{sec: coh}

Let $G$ be a topological group. A \textbf{Banach} $G$-module $E$ is a Banach space on which $G$ acts by linear isometries. The \textbf{continuous bounded cohomology} $\Hcb^*(G, E)$ of $G$ with coefficient in $E$ is the cohomology of the complex
$$0 \longrightarrow \bC(G, E)^G\overset{d}\longrightarrow \bC(G^2, E)^G\overset{d}\longrightarrow\bC(G^3,E)^G\overset{d}\longrightarrow \; ...$$
of $G$-equivariant bounded continuous maps from $G^{*+1}$ to $E$, where the coboundary operator $d: \bC(G^{n},E)^G\to\bC(G^{n+1},E)^G$ is defined by 
\begin{equation}\label{eq: coboundary}
    df(x_0,...,x_n) :=\sum_{j=0}^n(-1)^jf(x_0,...,\widehat{x_j},...,x_n),
\end{equation}
In this section, we collect necessary facts from continuous bounded cohomology. We refer unfamiliar readers to \cite{Mon01} for a detailed account of the theory of continuous bounded cohomology. 

Let $Z$ be a regular $G$-space (see \cite{Mon01}, Section 2.1), and let $E$ be a coefficient $G$-module (see \cite{Mon01}, Section 1.2). Let $L_\mathrm{w*}^\infty(Z, E)$ (resp. $L^\infty(Z,E)$) be the Banach space which consists of classes of essentially bounded measurable maps with respect to the Borel $\sigma$-algebra of the weak*-topology (resp. the norm topology) on $E$. Note that if $E$ is a separable Hilbert space, then $L_\mathrm{w*}^\infty(Z, E)=L^\infty(Z,E)$. In the following theorem by Burger and Monod, a \textbf{regular amenable $G$-space} is a {regular $G$-space} on which the $G$-action is amenable in the sense of Zimmer (\cite{Zim84}, Section 4.3), and the boundary operator $d$ is defined in the same way as (\ref{eq: coboundary}).

\begin{thm}[\cite{BM02b}, Theorem 2; \cite{Mon01}, Theorem 7.5.3]\label{thm: BM02b}
Let $G$ be a locally compact second countable group. If $Z$ is an amenable regular $G$-space and $E$ is a coefficient $G$-module, then the continuous bounded cohomology $\Hcb^*(G, E)$ is canonically isometrically isomorphic to the cohomology of the complex
$$
0 \longrightarrow L_{\mathrm{w*}}^\infty(Z, E)^G\overset{d}\longrightarrow L_{\mathrm{w*}}^\infty(Z^2, E)^G\overset{d}\longrightarrow L_{\mathrm{w*}}^\infty(Z^3, E)^G\longrightarrow \; ...
$$
of $G$-equivariant cochains in $L_{\mathrm{w*}}^\infty(G^{*}, E)$. The same statement holds for the subcomplex $(L_{\mathrm{w*},\alt}^\infty(Z^{*}, E)^G,d)$ of alternating cochains.
\end{thm}

We will use the following special case of Theorem \ref{thm: BM02b}:
\begin{cor}\label{cor: reso_bd}
Let $T=T_q$ be the $(q+1)$-regular tree and let $E$ be a coefficient $\aut(T)$-module. Then the continuous bounded cohomology $\Hcb^*(\aut(T), E)$ is canonically isometrically isomorphic to the cohomology of the complex
$$
0 \longrightarrow L_{\mathrm{w*,\alt}}^\infty(\bd T, E)^{\aut(T)}\overset{d}\longrightarrow L_{\mathrm{w*},\alt}^\infty((\bd T)^2, E)^{\aut(T)}\overset{d}\longrightarrow L_{\mathrm{w*}, \alt}^\infty((\bd T)^3, E)^{\aut(T)}\longrightarrow \; ...
$$
of $\aut(T)$-equivariant, alternating cochains in $L_{\mathrm{w*}}^\infty((\bd T)^{*}, E)$.
\begin{proof}
Let $\gamma\in \bd T$ and let $P\leq \aut(T)$ be the stabilizer of $\gamma$. Then $\bd T$ can be identified with the homogeneous space $G/P$ as regular $\aut(T)$-spaces. Since $P$ is amenable (\cite{F-TN91}, Section I.8), the action of $G$ on $G/P$ is amenable in the sense of Zimmer (\cite{Zim84}, Proposition 4.3.2). The proof is complete by Theorem \ref{thm: BM02b}.
\end{proof}
\end{cor}

We also recall that the continuous bounded cohomology $\Hcb^*(G, E)$ of a topological group $G$ vanishes in all positive degrees whenever $E$ is a relatively injective $G$-module (\cite{Mon01}, Proposition 7.4.1). This leads to the following vanishing result for $L^\infty$-inductions:

\begin{lem}\label{lem: rel_inj}
Let $G$ be a unimodular locally compact second countable group and $H$ be a compact open subgroup. Let $(\omega, V)$ be a unitary representation of $H$. Then the Banach $G$-module $L^\infty\ind_H^G \omega$ is relatively injective. In particular, one has $\Hcb^n(G, L^\infty\ind_H^G \omega)=0$ for all $n\geq 1$.
\begin{proof}
   Since the Banach $G$-module $L^\infty(G,V)$ with the left regular representation is relatively injective (\cite{Mon01}, Corollary 4.4.5), it is enough to show that the inclusion $\cI^\infty V\xhookrightarrow{i} L^\infty(G,V)$ admits a $G$-equivariant left inverse of norm $1$ by Proposition 4.3.1 in \cite{Mon01}. To this end, define
   $$r(f)(x):= \int_H\omega(h)f(xh)\,d\mu_H(h)\in V $$
   for each $f\in L^\infty(G, V)$ and $x\in G$, where $\mu_H$ is the Haar measure on $H$ with $\mu_H(H)=1$. Then for any $h_0\in H$ and $x\in G$, we have
   \begin{align*}
       r(f)(xh_0) &= \int_H \omega(h)f(xh_0h)d\mu_H(h)
       = \int_H \omega(h_0^{-1}h_0h)f(xh_0h)d\mu_H(h)\\
       &= \int_H \omega (h_0^{-1}h)f(xh)d\mu_H(h)
       = \omega (h_0^{-1}) \int_H\omega(h)f(xh)\,d\mu_H(h)\\
       &= \omega(h_0^{-1})r(f)(x).
   \end{align*}
Thus, we have a well-defined $G$-equivariant map $r: L^\infty(G,V)\to \cI^\infty V$. It is evident that $\norm{r(f)}_\infty \leq \norm{f}_\infty$ for all $f\in L^\infty(G,V)$ and $r\circ i=\id_{\cI^\infty V}$, which completes the proof.
\end{proof}
\end{lem}

Finally, we recall the following vanishing result for degree 1:

\begin{prop}[\cite{Mon01}, Proposition 6.2.1]\label{prop: degree_1}
Let $G$ be a topological group. Then 
$$\Hcb^1(G,E)=0$$
for any reflexive Banach $G$-module $E$.
\end{prop}
Thus, it is enough for us to work with degree 2 and beyond. 


\section{Almost Everywhere Identities}\label{sec: ae}

The cochain complex in Corollary \ref{cor: reso_bd} consists of $G$-equivariant and alternating function classes. To allow us to work with actual $G$-equivariant and alternating functions rather than function classes, we collect in this section a few almost everywhere identities. 

\begin{lem}[\cite{Zim84}, Proposition B.5]\label{lem: a.e.zimmer}
Let $G$ be a locally compact second countable group. Suppose that $(X,\mu)$ is a standard Borel $G$-space with quasi-invariant measure, and $Y$ is a standard Borel $G$-space. Suppose that $f: X\to Y$ is a Borel measurable function such that for any $g\in G$, one has $f(g\cdot x) = g\cdot f(x)$ for almost all $x\in X$, then there exists a $G$-invariant conull subset $X_0\subset X$ and a Borel measurable map $\wt{f}: X_0\to Y$ such that 
\begin{enumerate}[nosep]
    \item $\wt{f}$ is $G$-equivariant, and
    \item $f=\wt{f}$ almost everywhere on $X_0$.
\end{enumerate} 
\end{lem}

Let $X$ be a topological space and let $\mu$ be a measure defined on the Borel $\sigma$-algebra $\cB_X$ of $X$. We say that a map $f: X\to E$ between $X$ and a dual Banach space $E$ is \textbf{weak}*-\textbf{Borel}-\textbf{measurable} if it is Borel measurable when $E$ is equipped with the Borel $\sigma$-algebra of its weak*-topology. We say that $f$ is \textbf{weak}*-\textbf{measurable} if it is a measurable map from the completion $(X,\overline{\cB_X},\bar{\mu})$ of $(X,\cB_X,\mu)$ to the Borel $\sigma$-algebra of the weak*-topology of $E$. 

\begin{lem}\label{lem: a.e.}
Let $G$ be a locally compact second countable group and let $(X,\mu)$ be a standard Borel $G$-space with quasi-invariant measure. Let $E$ be a coefficient $G$-module. Then for any class $\alpha\in L_{\mathrm{alt,w*}}^\infty(X^{n}, E)^G$, there exists a bounded weak*-Borel-measurable representative $f_\alpha: X^n\to E$ and a conull subsets $X_0\subset X^n$ such that 
\begin{enumerate}[nosep]
    \item $X_0$ is invariant under both the action of $G$ and the action of $\Sym(n)$ by permuting the $n$-coordinates;
    \item $f_\alpha|_{X_0}$ is alternating and $G$-equivariant.
\end{enumerate}
\begin{proof}
The symmetric group $\Sym(n)$ acts on $X^n$ by permuting coordinates, and it has an action on $E$ by multiplying the signature character $\sgn:\Sym(n)\to \{\pm 1\}$. A map $f: X^n\to E$ is alternating if and only if it is $\Sym(n)$-equivariant. 

Let $\alpha\in L_{\mathrm{alt,w*}}^\infty(X^{n}, E)^G$ and let $f: X^n\to E$ be a bounded representative of $\alpha$. Then the image of $f$ is contained in the closed ball $B_\alpha$ of radius $\norm{\alpha}_\infty$ in $E$. Since $E$ is a coefficient $G$-module, the ball $B_\alpha$ is a standard Borel $G$-space under the weak*-topology. Thus, the weak*-measurable map $f$ is almost everywhere equal to a weak*-Borel-measurable map $X^n\to B_\alpha$ (see, for example, Proposition 2.2.3 in \cite{Coh}, along with the fact that any two uncountable standard Borel spaces are Borel isomorphic (Theorem 3.3.13 in \cite{Sri91})), still denoted by $f$.

Since the action of the symmetric group $\Sym(n)$ on $X^n$ and $E$ commutes with the corresponding $G$ actions, we may view $X^n$ and $E$ as $(\Sym(n)\times G)$-spaces. For any $(\sigma, g)\in \Sym(n)\times G$, we have
\begin{enumerate}[noitemsep]
    \item for every $\sigma\in \Sym(n)$, there exists a conull subset $X_\sigma\subset X^n$ such that $f(\sigma x) = \sgn(\sigma)\cdot f(x)$ for all $x\in X_\sigma$, and
    \item for every $g\in G$, there exists a conull subset $X_g\subset X^n$ such that $f(g\cdot x)=g\cdot f(x)$ for all $x\in X_g$.
\end{enumerate}
Thus, for any $x$ in the conull subset $X_\sigma\cap \sigma^{-1}X_{g}$, one has
$$f(g\cdot\sigma x)=g\cdot f(\sigma x) = \sgn(\sigma)\cdot g\cdot f(x).$$
Thus, Lemma \ref{lem: a.e.zimmer} applies and we get a ($\Sym(n)\times G$)-invariant conull subset $X_0\subset X^n$ and a ($\Sym(n)\times G$)-equivariant weak*-Borel-measurable map $\wt{f}:X_0\to B_\alpha\subset E$ that is equal to $f$ almost everywhere on $X_0$. Finally, we may extend $\wt{f}$ to a weak*-Borel-measurable map $f_\alpha: X^n\to B_\alpha\subset E$ (see, for example, \cite{Sho83}), and the proof is complete.
\end{proof}
\end{lem}

\begin{rmk}\label{rmk: a.e. orbits1}
\emph{In the above Lemma \ref{lem: a.e.}, if $X_1\subset X$ is a $G$-orbit of positive measure, then for any two pairs $(f_\alpha,X_0)$ and $(f_\alpha',X_0')$ that satisfies the two conditions in Lemma \ref{lem: a.e.}, one has $X_1\subset X_0\cap X_0'$ by the $G$-invariance of $X_0$ and $X_0'$.
Moreover, we have $f_\alpha|_{X_1}=f_\alpha'|_{X_1}$. Indeed, since the two maps $f_\alpha$ and $f_\alpha'$ agree almost everywhere and $X_1$ is of positive measure, there exists some $x\in X_1$ such that $f_\alpha(x)=f_\alpha'(x)$. It follows that $f_\alpha(gx)=f_\alpha'(gx)$ for all $g\in G$ since $f_\alpha'|_{X_0'}$ are $G$-equivariant.}
\end{rmk}

\begin{rmk}\label{rmk: a.e. orbits}
\emph{If $X = (\bd T)^n$ and $G=\aut(T_q)$, then any $G$-orbit in the conull subset $D_n(\bd T)\subset X$ contains a subset of positive measure. Indeed, if $\gamma=(\gamma_0,...,\gamma_{n-1})\in D_n(\bd T)$ and $m\in T$ is some vertex, take $N_\gamma:= \max_{i\ne j} (\gamma_i, \gamma_j)_m$. Then the $G$-orbit of $\gamma$ contains the neighborhood
$$
U:=\{\gamma'=(\gamma'_0,...,\gamma_{n-1}')\in D_n(\bd T): (\gamma_j',\gamma_j)_m > N_\gamma\},
$$
which has positive measure. Thus, we may always take $X_0$ to be the conull subset $D_n(\bd T)$ by Remark \ref{rmk: a.e. orbits1}.} 
\end{rmk}

Finally, we record the following two facts for later use:

\begin{cor}\label{cor: a.e.}
Let $E$ be a coefficient $G$-module. If any bounded weak*-Borel-measurable map $f: (\bd T)^{n+1}\to E$ that is $G$-equivariant and alternating on $D_{n+1}(\bd T)$ vanishes almost everywhere, then $\Hcb^n(\aut(T),E)=0$.
\begin{proof}
Let $\alpha\in L_{\alt,\mathrm{w}^*}^\infty((\bd T)^{n+1}, E)$ be a cochain, then by Lemma \ref{lem: a.e.} and Remark \ref{rmk: a.e. orbits}, there is a bounded weak*-Borel measurable representative $f_\alpha: (\bd T)^{n+1}\to E$ such that $f_\alpha$ is alternating and $G$-equivariant on $D_{n+1}(\bd T)$. Since $f_\alpha$ vanishes almost everywhere, we have $\alpha = 0$ and the conclusion follows from Corollary \ref{cor: reso_bd}.
\end{proof}
\end{cor}

\begin{cor}\label{cor: a.e. 2}
Let $n\geq 1$ and let $E$ be a coefficient $G$-module. Denote by $\cF_\alt((\bd T)^{n+1},E)^G$ the vector space of all $G$-equivariant alternating maps from $(\bd T)^{n+1}\to E$. Then any map $f\in \cF_\alt((\bd T)^{n+1},E)^G$ is weak*-continuous on $D_{n+1}(\bd T)$. Moreover, there is a vector space isomorphism
$$\cF_\alt^\infty((\bd T)^{n+1},E)^G\overset{\simeq}\longrightarrow L_{\alt,\mathrm{w}^*}^\infty((\bd T)^{n+1},E)^G,$$
where $\cF_\alt^\infty((\bd T)^{n+1},E)^G\subset \cF_\alt((\bd T)^{n+1},E)^G$ is the subspace of bounded maps.
\begin{proof}
We first show that every map $f\in \cF_\alt((\bd T)^{n+1},E)^G$ is weak*-continuous on $D_{n+1}(\bd T)$: Let $\gamma\in D_{n+1}(\bd T)$ and set $v:=f(\gamma)\in E$. Let $O_v$ be a neighborhood of $v$. Since $E$ is a coefficient $G$-module, the map $\varphi_v: g\mapsto gv$ is weak*-continuous. Choose a neighborhood $U_e\subset G$ of identity such that $\varphi_v(U_e)\subset O_v$, and consider $U_e\cdot\gamma:=\{g\gamma: g\in U_e\}$ which contains a neighborhood of $\gamma$. Then $f(U_e\cdot\gamma)\subset O_v$ by $G$-equivariance. This shows that $f|_{D_{n+1}(\bd T)}$ is weak*-continuous. In particular, the restriction $f|_{D_{n+1}(\bd T)}$ defines a class $c_f\in L_{\alt,\mathrm{w}^*}^\infty((\bd T)^{n+1},E)^G$ for any $f\in \cF_\alt^\infty((\bd T)^{n+1},E)^G$

We now show that the linear map $f\mapsto c_f$ is an isomorphism. To see that it is injective, suppose that $c_f=0$. Then $f=0$ almost everywhere on $D_{n+1}(\bd T)$. Since any $G$-orbit on $D_{n+1}(\bd T)$ is of positive measure (Remark \ref{rmk: a.e. orbits}), we have that $f=0$ on $D_{n+1}(\bd T)$, and hence on the whole $(\bd T)^{n+1}$ since $f$ is alternating. For surjectivity, let $c\in L_{\alt,\mathrm{w}^*}^\infty((\bd T)^{n+1},E)^G$ and let $f_c$ be its unique representative that is alternating and $G$-equivariant on $D_{n+1}(\bd T)$ (Remark \ref{rmk: a.e. orbits} and Lemma \ref{lem: a.e.}). Then the map $f\in \cF_\alt^\infty((\bd T)^{n+1},E)^G$ defined by $f=f_c$ on $D_{n+1}(\bd T)$ and zero elsewhere satisfies $c_f=c$. 
\end{proof}
\end{cor}


\section{Proof of Theorem \ref{thm: A} and \ref{thm: B}}\label{sec: pf}

\subsection{The Flipping Lemma}

Let $T=T_q$ with $q\geq 2$ and let $(\pi, \cH_\pi)$ be an irreducible representation of $G:=\aut(T)$. As seen in Section \ref{sec: uni_rep}, the Hilbert space $\cH_\pi$ is often, roughly speaking,  an $L^2$-space on a collection $\mathcal{S}_\pi$ of subtrees. One of the key observations in our proof of the vanishing of $\Hcb^*(G, \cH_\pi)$ is that the space $L_\alt^\infty((\bd T)^{n+1}, \cH_\pi)^G$ is usually small and, in many cases, trivial: For a map $\alpha: (\bd T)^{n+1}\to \cH_\pi$, a tuple $(\gamma_0,...,\gamma_n)\in (\bd T)^{n+1}$, and a subtree $S\in \mathcal{S}_\pi$, we may often choose a suitable elliptic element $h\in \aut(T)$ that permutes $(\gamma_0,...,\gamma_n)$ with a negative sign and fixes $S$ pointwise. Then $\alpha(\gamma_0,...,\gamma_n)(S)$ is forced to be 0 by the $G$-equivariance and the alternating conditions of $\alpha$. In this section, we establish a condition for the existence of such $h$'s, which immediately leads to the vanishing of $\Hcb^*(G, \cH_\pi)$ when $\pi$ is special. 

\begin{defn}
\emph{Let $\gamma = (\gamma_0,\gamma_1,\gamma_2)\in D_3(\bd T)$ be a triple of distinct elements with median $m_\gamma$, and let $S$ be a subtree of $T$. We say that $S$ \textbf{misses} $\gamma$ if there are at least two distinct indices $i,j\in \{0,1,2\}$ such that }
$$(\gamma_i, x)_{m_\gamma} = (\gamma_j,x)_{m_\gamma}=0$$
\emph{for any vertex $x\in S$. Otherwise, we say that $S$ \textbf{hits} $\gamma$.}
\end{defn}

Observe that if $S$ is a finite complete subtree, then it hits $\gamma\in D_3(\bd T)$ if and only if the median of $\gamma$ is a vertex of degree $q+1$ in $S$. In other words, a finite complete subtree hits a triple if and only if it "passes through" its median. In particular, if $(\pi, \cH_\pi)$ is an irreducible unitary representation of $\aut(T)$, then there are at most finitely many elements in $\cM_\pi$ that hit $\gamma$. 

\begin{lem}[The Flipping Lemma]\label{lem: flip}
Let $n\geq 2$ and let $\gamma= (\gamma_0,...,\gamma_n)\in D_{n+1}(\bd T)$ be a tuple of distinct elements in $\bd T$. Suppose that $S$ is a (possibly empty) subtree in $T$ that misses the triple $(\gamma_0,\gamma_1,\gamma_2)$, then there are two distinct indices $i,j\in \{0,...,n\}$ and an element 
$$h_{ij}\in G(S)\cap \bigcap_{k\ne i,j} \stab_G(\gamma_k)$$
such that $h_{ij}\cdot\gamma_i=\gamma_j$ and $h_{ij}\cdot \gamma_j = \gamma_i$.
\begin{proof}
Let $m_\gamma$ be the median of $\gamma_0, \gamma_1$ and $\gamma_2$. Since $S$ misses $(\gamma_0,\gamma_1,\gamma_2)$, we may assume without loss of generality that $(\gamma_1, x)_{m_\gamma} = (\gamma_2,x)_{m_\gamma}=0$ for all $x\in S$. Consider the subset
$$J_0 :=\{j: (\gamma_j,\gamma_1)_{m_\gamma}>0\}\cup\{j:(\gamma_j,\gamma_2)_{m_\gamma}>0\}\subset\{1,...,{n}\}.
$$
Note that $\{1,2\}\subset J_0$ and $(\gamma_l,\gamma_0)_{m_\gamma}=0$ for all $l\in J_0$. Choose two indices $i,j\in J_0$ such that 
$$(\gamma_i, \gamma_j)_{m_\gamma} = \max_{k,l\in J_0,\, k\ne l}\,(\gamma_k,\gamma_l)_{m_\gamma}.$$
and denote by $m_{\gamma}'$ the median of $m_\gamma, \gamma_i$ and $\gamma_j$. Then by our choice of $i$ and $j$, we have that 
$(\gamma_k,\gamma_i)_{m_\gamma'}=(\gamma_k,\gamma_j)_{m_\gamma'}=0$
for all $k\in\{0,...,n\}\setminus\{i,j\}$, and
$(x,\gamma_i)_{m_\gamma'}=(x,\gamma_j)_{m_\gamma'}=0$
for all $x\in S$. The existence of the desired $h_{ij}$ is now evident (see Figure \ref{fig: flip_cusp} for an illustration). 
\end{proof}
\end{lem}

\begin{figure}[!htb]
\centering
\setlength{\unitlength}{0.5cm}

\begin{picture}(10,10)
\put(5,4){{\color{black}\circle*{0.2}}\hbox{\kern4pt\texttt{$m_\gamma$}}}
\put(5,4){\vector(0,1){5.5}}  
\put(5,9.5){\makebox(0,0)[l]{\ $\gamma_0$}}

\put(5,6){\vector(1,0){4}}
\put(9,6){\makebox(0,0)[l]{\ $\gamma_4$}}

\put(5,4){\vector(-1,-0.8){4.5}}
\put(0,0){\makebox(0,0)[l]{\ $\gamma_1$}}
{
\put(3,2.4){\line(1,-0.8){0.99}}
\put(4,1.6){\color{blue}\vector(1,-0.8){1.5}}
\put(5.4,0){\color{blue}\makebox(0,0)[l]{\ $\gamma_3$}}
}

\put(4,1.6){{\color{blue}\circle*{0.2}}\hbox{\kern4pt\texttt{$m_\gamma'$}}}
{
\put(4,1.6){\color{blue}\vector(-1,-0.8){1.5}}
\put(2.5,0){\color{blue}\makebox(0,0)[r]{\ $\gamma_5$}}
}
\put(3.0,0.5){\color{blue}\makebox(0,0)[l]{\ $\xleftrightarrow[]{\;h_{35}\;}$}}
\put(5,4){\vector(1,-0.8){4.5}}
\put(9.5,0){\makebox(0,0)[l]{\ $\gamma_2$}}
\linethickness{0.7mm}
\put(5,7){\color{red}\line(0,1){0.5}}
\put(5,7){\color{red}\line(-1,0){0.5}}
\put(5,7){\color{red}\line(0,-1){0.5}}
\put(3.5,7){\color{red} \makebox(0,0)[l]{\ $S$}}
\end{picture}
\caption{\label{fig: flip_cusp} The element $h_{35}\in G(S)$ flips $\gamma_3$ and $\gamma_5$ and stabilizes all other $\gamma_k$'s.}
\end{figure}
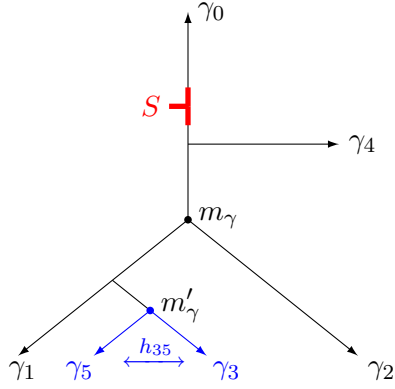

As a toy example, we establish the vanishing of $\Hcb^*(G, \cH_\pi)$ when $\pi$ is special.

\begin{thm}\label{thm: special}
Fix $n\geq 2$. Let $\alpha: (\bd T)^{n+1}\to \ell^2(\wt{\mathcal{E}_T}, 
 \C)$ be any map that is alternating and $G$-equivariant on $D_{n+1}(\bd T)$. Then $\alpha = 0$ on $D_{n+1}(\bd T)$. In particular, one has $\Hcb^n(G, \cH_\pi)=0$ for all $n\geq 1$ for any unitary special representations $(\pi, \cH_\pi)$.
 \begin{proof}
Let $\gamma =(\gamma_0,...,\gamma_n)\in D_{n+1}(\bd T)$ and let $e = e_{xy}\in \wt{\mathcal{E}_T}$ be an oriented edge. Then $e$ cannot hit the triple $(\gamma_0,\gamma_1,\gamma_2)$. By Lemma \ref{lem: flip}, there are two indices $i,j\in \{0,...,n\}$ and an element $h_{ij}\in \stab_G(e)\cap\bigcap_{k\ne i,j} \stab_G(\gamma_k)$ such that $h_{ij}(\gamma_i)=\gamma_j$ and $h_{ij}(\gamma_j)=\gamma_i$. Since $\alpha$ is $G$-equivariant and alternating on $D_{n+1}(\bd T)$, one has
$$\alpha(\gamma)(e) = \alpha(h_{ij}\gamma)(h_{ij}e)=-\alpha(\gamma)(e).$$
It follows that $\alpha(\gamma)=0$. Since $\cH_\pi$ is a submodule in $\ell^2(\wt{\cE_T}, \C)$ when $\pi$ is special, the vanishing of $\Hcb^n(G,\cH_\pi)$ for all $n\geq 1$ follows from Corollary \ref{cor: reso_bd}, along with \ref{prop: degree_1}.
 \end{proof}
\end{thm}

\subsection{The Spherical Cases}

Let $T=T_q$ with $q\geq 2$ and let $G=\aut(T)$. We adapt the notations from Section \ref{sec: sph_intro}. In particular, we fix a vertex $o\in T$ and let $K$ be the stabilizer of $o$. The goal of this section is to prove the following result: 

\begin{thm}\label{thm: spherical}
Let $z\in \C$ with $\mu(z)\in [-1,1]$. Then $\Hcb^n(G, \cH_z)=0$ for all $n\geq 1$. In other words, for any unitary spherical representation $(\pi,\cH_\pi)$, we have $\Hcb^n(G, \cH_\pi)=0$ for all $n\geq 1$.
\end{thm}

We start by dealing with the edge cases where $\mu(z)=\pm 1$. This is another direct application of the flipping lemma:

\begin{prop}\label{prop: re_z=0}
Let $z\in \C$ with $\mu(z)\in \{\pm 1\}$. Then $\Hcb^n(G, \cH_z)=0$ for all $n\geq 1$. 
\begin{proof}
If $\mu(z)=1$, the spherical function $\varphi_z$ is the constant function $\mathbf{1}$ and $(\pi_z, \cH_z)$ is the trivial representation. If $\mu(z)=-1$, then $\varphi_z(g) = (-1)^{d_T(o,go)}$ for all $g\in G$. In both cases, the space $\cH_z$ is one-dimensional and $K$-invariant.

Let $n\geq 2$, and let $\alpha: (\bd T)^{n+1}\to \cH_z$ be a bounded map that is $G$-equivariant and alternating on $D_{n+1}(\bd T)$. Let  $\gamma = (\gamma_0,...,\gamma_n)\in D_{n+1}(\bd T)$. Applying Lemma \ref{lem: flip} with $S=\varnothing$, there exists some $g = g_{ij}\in K$ such that $(g\gamma_i, g\gamma_j)=(\gamma_j,\gamma_i)$ for some $i,j\in \{0,...,n\}$ and $g\gamma_k = \gamma_k$ for all $k\notin \{i,j\}$. Since $\cH_z$ is $K$-invariant and $\alpha$ is alternating and $G$-equivariant, we have
$$
\alpha(\gamma) = g\cdot\alpha(\gamma) = \alpha(g\gamma_0,...,g\gamma_n) = -\alpha(\gamma).
$$
It follows that $\alpha = 0$. This completes the proof by Corollary \ref{cor: reso_bd} and Proposition \ref{prop: degree_1}.
\end{proof}
\end{prop}

\begin{rmk}
\emph{The vanishing of continuous bounded cohomology of $G$ with the trivial coefficient was already proved by Bucher and Monod in \cite{BM19}, where they showed that $\Hcb^*(H, \R)=0$ in all positive degrees for all closed subgroup $H\leq G$ acting strongly transitively on $\bd T$. Note that our proof of the previous proposition, while simpler, does not extend to all such closed subgroups.}
\end{rmk}

From now on, we assume that $\mu(z)\in (-1,1)$. Let $\cH$ be a Hilbert space, and let $\alpha: (\bd T)^{n+1}\to \cH$ be a bounded measurable function. Consider the bounded map $\cP\alpha: T^{n+1}\to \cH$ on $(n+1)$-tuples of vertices defined by
$$
\cP\alpha(x_0,...,x_n) := \frac{1}{\prod_{j=0}^n\mu_o(U(o,x_j))}\cdot \int_{U(o,x_0)\times...\times U(o,x_n)} \alpha(\gamma)\;d\mu_o^{\otimes ({n+1})}(\gamma).
$$
We call $\cP\alpha$ the \textbf{averaging transform} of $\alpha$. 

Now suppose that $\cH=\cH_z'$. Since the measure $\mu_o$ is invariant under $K$ and permutations of coordinates, the averaging transform $\cP\alpha$ is $K$-equivariant (respectively alternating) whenever $\alpha$ is $K$-equivariant (resp. alternating) on $D_{n+1}(\bd T)$. 


We will need the following two facts about the averaging transform:

\begin{lem}\label{cor: smooth}
Let $\alpha:  (\bd T)^{n+1}\to \cH_z'$ be a bounded measurable function that is $K$-equivariant on $D_{n+1}(\bd T)$. Then $\cP\alpha(x_0,...,x_n)\in \cK(\bd T)$ for any $x_0,...,x_n\in T$.
\begin{proof}
By Lemma \ref{lem: smooth}, it is enough to show that $\cP\alpha(x_0,...,x_n)$ is smooth. Choose 
$$H:= K\cap\bigcap_{j=0}^n\stab_G(x_j).$$
Then $H$ is a compact open subgroup of $G$. Since $\cP\alpha$ is $K$-equivariant, we have that
$$\pi_z'(h)\cdot\cP\alpha(x_0,...,x_n) = \cP\alpha(hx_0,...,hx_n)=\cP\alpha(x_0,...,x_n)$$
for all $h\in H$, and so $\cP\alpha(x_0,...,x_n)$ is smooth. 
\end{proof}
\end{lem}

\begin{lem}\label{lem: Poisson_inv}
Let $\cH$ be a Hilbert space, and let $\alpha: (\bd T)^{n+1}\to \cH$ be a bounded measurable function. Then for any $\varphi\in \cH$, there is a conull subset $X_\varphi\subset (\bd T)^{n+1}$ such that for all $\gamma = (\gamma_0,...,\gamma_n)\in X_\varphi$, one has
$$
\inprod{\alpha(\gamma), \varphi}_\cH = \lim_{t\to+\infty}\inprod{ \cP\alpha(\gamma_0(o,t),...,\gamma_n(o,t)),\; \varphi}_\cH\;.
$$
\begin{proof}
For $\gamma,\eta\in \bd T$, define $d_o(\gamma,\eta):=q^{-(\gamma,\eta)_o}$. Then $d_o$ is an ultrametric that realizes the topology on $\bd T$. Then the product metric on $(\bd T)^{n+1}$ defined by
$$d_{o,n+1}((\gamma_0,...,\gamma_{n}),(\eta_0,...,\eta_n)):=\max_{j=0,...,n} d_o(\gamma_j,\eta_j)$$ 
is also an ultrametric. Denote by $B^{n+1}(\gamma, r)$ the open ball of radius $r>0$ centered around $\gamma\in (\bd T)^{n+1}$. Then for any $\gamma = (\gamma_0,...,\gamma_n)\in (\bd T)^{n+1}$, one has
\begin{align}\label{diff}
    \cP\alpha(\gamma_0(o,t),...,\gamma_n(o,t)) = \frac{1}{\mu_o^{\otimes(n+1)}(B^{n+1}(\gamma, q^{-(t-1)}))}\cdot \int_{B^{n+1}(\gamma, q^{-(t-1)})}\alpha \;d{\mu_o^{\otimes{(n+1)}}}.
\end{align}
Now for each $\varphi\in \cH$, consider the bounded measurable function $\alpha_\varphi: (\bd T)^{n+1}\to \C$ defined by $\gamma \mapsto \inprod{\alpha(\gamma), \varphi}_\cH$. Then by (\ref{diff}) and the Lebesgue differentiation theorem for ultrametric spaces (see Theorem 9.1 in \cite{Sim12}, which is based on section 2.8-2.9 in \cite{Fed69}), there is a conull subset $X_\varphi\subset(\bd T)^{n+1}$ such that for any $\gamma=(\gamma_0,...\gamma_n)\in X_\varphi$,
$$\alpha_\varphi(\gamma) = \lim_{t\to \infty}\cP\alpha_\varphi(\gamma_0(o,t),...,\gamma_n(o,t)).$$
The lemma now follows from the fact that 
$\cP\alpha_\varphi(x_0,...,x_n)=\inprod{\cP\alpha(x_0,...,x_n), \;\varphi}_\cH$ for all $x_0,...,x_n\in T$. 
\end{proof}
\end{lem}

\begin{proof}[Proof of Theorem \ref{thm: spherical}]
Let $z\in \C$ with $\mu(z)\in (-1,1)$ and fix $n\geq 2$. Let $\alpha: (\bd T)^{n+1}\to \cH'_z$ be a bounded measurable map that is alternating and $G$-equivariant on $D_{n+1}(\bd T)$. For each vertex $x\in T$, denote by $B_x\subset D_{n+1}(\bd T)$ the subset consists of all $\gamma = (\gamma_0,...,\gamma_n)$ such that the medium of $\gamma_0,\gamma_1,\gamma_2$ is $x$. Note that $B_x$ has positive measure for any $x\in T$. 

We first show that it is enough to find a subset $B_o'\subset B_o$ such that $\mu_o(B_o\setminus B_o')=0$ and $\alpha(\gamma)=0$ for all $\gamma\in B_o'$. Suppose that such $B_o'$ exists. For each vertex $x\in T$, pick $g_x\in G$ such that $g_xo=x$. Then $g_xB_o=B_x$ for all $x\in T$ and $\mu_o(g_x(B_o\setminus B_o'))=0$ since $\mu_o$ is quasi-invariant. Consider the conull subset
$X_0 := \bigsqcup_{x\in T} g_xB_o'$ in $(\bd T)^{n+1}$. Then for any $x\in T$ and $\gamma\in B_o'$, we have
$$\alpha(g_x\gamma) = \pi_z'(g_x)\alpha(\gamma) = 0$$
by the $G$-equivariance of $\alpha$. This shows that $\alpha=0$ on the conull subset $X_0$, which finishes the proof by Corollary \ref{cor: reso_bd} and Proposition \ref{prop: degree_1}.

It remains to find such $B_o'$. Choose a countable dense subset $\cD_z$ in the separable Hilbert space $\cH_z'$. By Lemma \ref{lem: Poisson_inv}, there exists a conull subset $B\subset (\bd T)^{n+1}$ such that for all $\gamma\in B$,
$$(\alpha(\gamma), \varphi)_z = \lim_{t\to+\infty}(\cP\alpha(\gamma_0(o,t),...,\gamma_n(o,t)),\; \varphi)_z$$
for all $\varphi\in \cD_z$. Set $B_o':= B\cap B_{o}$. Then by the density of $\cD_z$ in $\cH_z'$, it suffices to show that for all $\gamma = (\gamma_0,...,\gamma_n)\in B_o'$, one has $(\alpha(\gamma), \varphi)_z = 0$ for all $\varphi\in \cD_z$. In fact, we will show that
$$\cP\alpha(\gamma_0(o,t),...,\gamma_n(o,t))=0$$
for all $(\gamma_0,...,\gamma_n)\in B_o'$ and $t\in \N$. By Lemma \ref{cor: smooth}, the element $\cP\alpha(\gamma_0(o,t),...,\gamma_n(o,t))$ is in $\cK(\bd T)$. In particular, it is a map from $\bd T$ to $\C$. Take any $\eta\in \bd T$. Then at least two of the three Gromov products $(\gamma_0, \eta)_o$, $(\gamma_1, \eta)_o$ and $(\gamma_2, \eta)_o$ are $0$. By the flipping lemma (\ref{lem: flip}), there are two indices $i,j\in \{0,...,n\}$ and an element 
$$h_{ij}\in K\cap\stab_G(\eta)\cap\bigcap_{k\ne i,j}\stab_G(\gamma_k)$$
such that $h_{ij}\cdot{\gamma_j} = \gamma_i$ and $h_{ij}\cdot\gamma_i=\gamma_j$. Since $d_T(\gamma_k(o,t),o)=t$ for all $k\in \{0,...,n\}$ and $t\in \N$, we have $h_{ij}\cdot\gamma_i(o,t)=h_{ij}\cdot\gamma_j(o,t)$ and vice versa for all $t\in \N$. Finally, since $\cP\alpha$ is $K$-equivariant, we have that
\begin{align*}
  \cP\alpha(\gamma_0(o,t),...,\gamma_n(o,t))(\eta)&=-\cP\alpha(h_{ij}\cdot\gamma_0(o,t),...,h_{ij}\cdot\gamma_n(o,t))(h_{ij}\cdot\eta)\\
  &=-\cP\alpha(\gamma_0(o,t),...,\gamma_n(o,t))(\eta).  
\end{align*}
for all $t\in \N$. It follows that $\cP\alpha(\gamma_0(o,t),...,\gamma_n(o,t))=0$ for all $t\in \N$.
\end{proof}

\subsection{The Cuspidal Cases}

\subsubsection{Vanishing of Bounded Cohomology in Higher Degrees}

Let $T=T_q$ with $q\geq 2$ and let $G=\aut(T)$. The goal of this section is to prove that the continuous bounded cohomology $\Hcb^n(G, \cH_\pi)$ vanishes for all $n\geq 3$ when $(\pi, \cH_\pi)$ is a unitary cuspidal representation of $\aut(T)$. In fact, we will prove the following slightly more general statement:

\begin{thm}\label{prop: vanishing_cusp}
Let $S_0$ be a finite complete subtree of $T$, and let $(\omega,V)$ be an irreducible unitary representation of $\aut(S_0)$. Then
$$\Hcb^n(\aut(T), \ind_{S_0}\,\omega) = 0$$
for all $n\geq 3$. 
\end{thm}

Write $\pi = \ind_{S_0} \omega$. Then we may identify $\cM_\pi$ with the quotient $G/\wt{G}(S_0)$. Fix a set-theoretic section $s: G/\wt{G}(S_0)\simeq \cM_\pi\to G$ that sends $S_0$ to the identity $e_G$ of $G$. As introduced in Section \ref{sec: cuspidal_intro}, the representation $\pi$ can be described by the $G$-module $\cH_\pi:= \ell^2(\cM_\pi, V)$, with the action of $G$ defined by
$$(\pi(g)f)(S) = \wt{\omega}(s(S)^{-1}gs(g^{-1}S))f(g^{-1}S)$$
for all $g\in G$, $f\in \cH_\pi$ and $S\in \cM_{\pi}$.

We will need the following essential observation, which is another application of the flipping lemma (\ref{lem: flip}):

\begin{lem}\label{lem: finite_supp}
Let $n\geq 3$. Suppose that $\alpha: D_n(\bd T)\to\ell^{\infty}(\cM_\pi,V)$ is a $G$-equivariant and alternating map, then for any $\gamma = (\gamma_0,...,\gamma_{n-1})\in D_n(\bd T)$
$$\mathrm{supp}(\alpha(\gamma)) \subset \{S\in\cM_\pi: S\;\,\mathrm{hits}\;\, (\gamma_i,\gamma_j,\gamma_k)\}$$
for all distinct $i,j,k\in \{0,...,n-1\}$, where $m_\gamma(i,j,k)\in T$ denotes the median of $\gamma_i,\gamma_j$ and $\gamma_k$. In particular, the map $\alpha(\gamma)$ is finitely supported and is an element in $\ell^2(\cM_\pi, V)$. 
\begin{proof}
Without loss of generality, take $(i,j,k)=(0,1,2)$ and write $m_\gamma:=m_\gamma(0,1,2)$. Let $m_\gamma\in T$ be the median of $\gamma_0,\gamma_1$, and $\gamma_2$, and let $S$ be a finite complete subtree in $\cM_\pi$ that misses $(\gamma_0,\gamma_1, \gamma_2)$. By Lemma \ref{lem: flip}, there are two distinct indices $i,j\in \{0,...,n\}$ and an element 
$$h_{ij}\in G(S)\cap \bigcap_{k\ne i,j} \stab_G(\gamma_k)$$
such that $h_{ij}\cdot\gamma_i=\gamma_j$ and $h_{ij}\cdot \gamma_j = \gamma_i$. Since $\alpha$ is $G$-equivariant and alternating, we have
\begin{align*}
    \alpha(\gamma)(S) =  -\alpha(h_{ij}\gamma)(h_{ij}S) = -\wt{\omega}(s(S)^{-1}h_{ij}s(S))\cdot \alpha(\gamma)(S).
\end{align*}
But observe that $s(S)^{-1}h_{ij}s(S)\in G(S_0)$, and so $\wt{\omega}(s(S)^{-1}h_{ij}s(S)) = \id_V$. It follows that $\alpha(\gamma)(S)=0$.
\end{proof}
\end{lem}

\begin{proof}[Proof of Theorem \ref{prop: vanishing_cusp}]
Fix $n\geq 3$ and let $c\in L^\infty_{\alt}((\bd T)^{n+1}, \ell^2(\cM_\pi, V))^G$ be a bounded cocycle. Since $\ell^2(\cM_\pi, V)$ is a submodule of $\ell^\infty(\cM_\pi, V)$, the function class $c$ can be seen as an element in $L^\infty_{\alt, \mathrm{w}^*}((\bd T)^{n+1}, \ell^\infty(\cM_\pi, V))^G$. By Lemma \ref{lem: rel_inj}, we have $\Hcb^n(G, \ell^\infty(\cM_\pi, V)) = 0$ for all $n\geq 1$, so the function class $c=d\alpha$ is a coboundary in $ L^\infty_{\alt}((\bd T)^{n+1}, \ell^\infty(\cM_\pi, V))^G$. By Lemma \ref{lem: a.e.} and Remark \ref{rmk: a.e. orbits}, there is a weak*-Borel-measurable representative $(\bd T)^n\to \ell^\infty(\cM_\pi, V)$ of $\alpha$, still denoted by $\alpha$, such that $\alpha|_{D_n(\bd T)}$ is bounded, alternating and $G$-equivariant. The restriction $\alpha|_{D_n(\bd T)}$ ranges in $\ell^2(\cM_\pi, V)$ by Lemma \ref{lem: finite_supp} and is measurable as a map $D_n(\bd T)\to \ell^2(\cM_\pi, V)$.

We now show that $\alpha(D_n(\bd T))$ is bounded in $\ell^2(\cM_\pi, V)$. Let $\gamma=(\gamma_0,...,\gamma_{n-1})\in D_n(\bd T)$ and fix any vertex $x\in T$. Define
$$\cM_\pi(\gamma):= \{S\in \cM_\pi: S \;\,\mathrm{hits}\;\,(\gamma_0,\gamma_1,\gamma_2)\}$$
and let $C := |M_\pi(\gamma)|\in \N.$
Note that $C$ is independent of the choice of $\gamma$. Let $m_\gamma$ be the median of $\gamma_0,\gamma_1$, and $\gamma_2$ in $\gamma=(\gamma_0,...,\gamma_{n-1})\in D_n(\bd T)$, then
\begin{align*}
    \norm{\alpha(\gamma)}_2^2 = \sum_{S\in \cM_\pi(\gamma)}\norm{\alpha(\gamma)(S)}_V^2\leq\sum_{S\in \cM_\pi(\gamma)}\norm{\alpha(\gamma)}_\infty^2=C\cdot\norm{\alpha(\gamma)}_\infty^2\leq C\cdot\norm{\alpha}_\infty^2
\end{align*}
for any $\gamma\in D_n(\bd T)$. It follows that $\alpha(D_n(\bd T))$ is bounded in $\ell^2(\cM_\pi, V)$. 

We have shown that the map $\alpha$ defines a bounded cochain in $L_{\mathrm{alt}}^\infty((\bd T)^n, \ell^2(\cM_\pi,V))^G$. The proof is now complete since $c = d\alpha$.
\end{proof}

\begin{cor}\label{cor: cuspidal_vanish}
Let $(\pi, \cH_\pi)$ be a unitary cuspidal representation of $G$. Then $\Hcb^n(G, \cH_\pi)=0$ for all $n\geq 3$. 
\begin{proof}
This follows directly from Theorem \ref{prop: vanishing_cusp} and the classification of unitary cuspidal representations by Ol'shanski (Theorem \ref{thm: class_cuspidal}).
\end{proof}
\end{cor}

\subsubsection{Hommage aux Centip\`edes}\label{subsec: centipede}

Let $T=T_{q}$ and let $G=\aut(T)$. It remains to compute $\Hcb^2(G, \cH_\pi)$ for unitary cuspidal representations $(\pi, \cH_\pi)$.

\begin{defn}\label{def: heads}
\emph{Let $S$ be a finite complete subtree with $\diam(S)\geq2$. A \textbf{head} of $S$ is the minimal subtree in $S$ that contains the vertices in $S\setminus S'$, where $S'$ is a maximal complete proper subtree of $S$. We say that $S$ is $k$-\textbf{headed} if it has $k$ heads, or equivalently, if it has $k$ maximal complete proper subtrees.}
\end{defn}
Note that any finite complete subtree of diameter at least 2 contains at least 2 heads. 
\begin{defn}\label{def: centipede}
\emph{A $2$-headed finite complete subtree $S$ of diameter $k\geq 3$ is called a $k$-\textbf{centipede}. Any finite complete subtree of diameter 2 is called a \textbf{spider}. We say that a centipede or a spider $S$ of diameter $k$ \textbf{lies on a geodesic} $L$ in $T$ if $\diam(S\cap L)=k$. }
\end{defn}

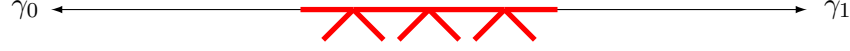
\begin{figure}[!htb]
    \centering
    \setlength{\unitlength}{0.5cm}
\begin{picture}(10,3)
\put(5,2){\vector(1,0){10}}
\put(5,2){\vector(-1,0){10}}
\put(-6.3,2){\makebox(0,0)[l]{\ $\gamma_0$}}
\put(15.2,2){\makebox(0,0)[l]{\ $\gamma_1$}}
\linethickness{0.7mm}
\put(5,2){\color{red}\line(1,0){2}}
\put(5,2){\color{red}\line(1,-1){0.8}}
\put(5,2){\color{red}\line(-1,-1){0.8}}
\put(5,2){\color{red}\line(-1,0){2}}
\linethickness{0.7mm}
\put(7,2){\color{red}\line(1,0){1.4}}
\put(7,2){\color{red}\line(1,-1){0.8}}
\put(7,2){\color{red}\line(-1,-1){0.8}}
\put(3,2){\color{red}\line(-1,0){1.4}}
\put(3,2){\color{red}\line(1,-1){0.8}}
\put(3,2){\color{red}\line(-1,-1){0.8}}
\end{picture}
    \caption{A $4$-centipede (thick red) that lies on $L(\gamma_0,\gamma_1)$ in the $4$-regular tree.}
    \label{fig:centipede}
\end{figure}

We now show that $\Hcb^2(G,\cH_\pi)\ne 0$ only if $\cM_\pi$ consists of centipedes or spiders. In fact, we will prove the following stronger statement:

\begin{thm}\label{prop: too_many_heads}
Fix $k\geq 3$. Let $(\pi, \cH_\pi)$ be a unitary cuspidal representation of $G$ such that $\cM_\pi$ consists of finite complete subtrees with at least $k$ heads. Then the subspace of cocycles
$$\mathcal{Z}L_{\mathrm{alt},\mathrm{w}^*}^\infty((\bd T)^k, \cH_\pi)^G:=\{c\in L_{\mathrm{alt},\mathrm{w}^*}^\infty((\bd T)^k, \cH_\pi)^G: dc=0\}$$
is trivial.
\end{thm}

As in the proof of Proposition \ref{prop: vanishing_cusp}, we fix an $S_0\in \cM_\pi$. Then $\pi = \ind_{S_0}\omega$ for some non-degenerate irreducible unitary representation $(\omega,V)$ of $\aut(S_0)$ (Theorem \ref{thm: class_cuspidal}). Identify $\cM_\pi$ with $G/\wt{G}(S_0)$ and fix a set-theoretic section $s: \cM_\pi\to G$ that sends $S_0$ to the identity of $G$. Then $\cH_\pi$ is identified by the $G$-module $\ell^2(\cM_\pi, V)$.

We first state the following observation:

\begin{lem}\label{lem: not_pass_head}
Let $k\geq 2$, and let $\alpha: X_0\subset(\bd T)^k\to \ell^\infty(\cM_\pi, V)$ be a $G$-equivariant map defined on a $G$-invariant subset $X_0$. Let $S\in \cM_\pi$ and let $x\in S$ be a vertex of degree $q+1$. Let $\gamma=(\gamma_0,...,\gamma_{k-1})\in X_0$. If there exists some maximal complete proper subtree $S'$ of $S$ such that 
$(S\setminus S')\cap [x,\gamma_j)=\varnothing$
for all $j = 0,...,k-1$, then $\alpha(\gamma)(S)=0$.
\begin{proof}
Since $X_0$ is $G$-invariant and the action of $G$ on $\cM_\pi$ is transitive (\cite{F-TN91}, Lemma 2.4 in Chapter III), the $G$-equivariance of $\alpha$ allows us to assume that $S = S_0$. For any $g\in G(S')$, the condition $(S_0\setminus S')\cap [x,\gamma_j)=\varnothing$ for all $j$ implies that $g\gamma_j=\gamma_j$ for all $j$. Since $\alpha$ is $G$-equivariant, we have that 
$$
\alpha(\gamma)(S_0) = \alpha(g\gamma)(gS_0) = \wt{\omega}(g)\cdot\alpha(\gamma)(S_0) = \omega(p_{S_0}(g))\cdot\alpha(\gamma)(S_0).
$$
for all $g\in G(S')$. So $\alpha(\gamma)(S_0)=0$ by the non-degeneracy of $\omega$.
\end{proof}
\end{lem}

\begin{proof}[Proof of Theorem \ref{prop: too_many_heads}]
Let $c\in \mathcal{Z}L_{\mathrm{alt},\mathrm{w}^*}^\infty((\bd T)^k, \cH_\pi)^G$. As in the proof of Proposition \ref{prop: vanishing_cusp}, we may view $c$ as a coboundary in $L_{\mathrm{alt},\mathrm{w}^*}^\infty((\bd T)^k, \ell^\infty(\cM_\pi, V))^G$. By Lemma \ref{lem: a.e.} and Remark \ref{rmk: a.e. orbits}, there exists a weak*-Borel-measurable bounded map $\alpha: (\bd T)^{k-1}\to \ell^\infty(\cM_\pi, V)$ 
such that $\alpha|_{D_k(\bd T)}$ is $G$-equivariant and $d\alpha$ represents $c$.

Let $\gamma=(\gamma_0,...,\gamma_{k-2})\in D_{n-1}(\bd T)$ and $S\in\cM_\pi$. Then for any $j=0,...,k-2$, there is at most one maximal complete proper subtree $S''$ of $S$ such that $[x,\gamma_j)\cap(S\setminus S'') \ne \varnothing$. Since $S$ has at least $k$ heads, the pigeonhole principle then guarantees the existence of a maximal complete proper subtree $S'$ of $S$ such that $[x,\gamma_j)\cap (S\setminus S') = \varnothing$ for all $j=0,...,k-2$. From Lemma \ref{lem: not_pass_head} it follows that $\alpha(\gamma)(S)=0$, and so $c=d\alpha=0$.
\end{proof}

\begin{cor}\label{cor: non_centipede_degree_2}
    Let $(\pi, \cH_\pi)$ be a unitary cuspidal representation such that $\cM_\pi$ consists of finite complete subtrees with three or more heads. Then $\Hcb^2(G, \cH_\pi)=0$.
\begin{proof}
Apply Theorem \ref{prop: too_many_heads} with $k=3$.
\end{proof}
\end{cor}

Corollary \ref{cor: non_centipede_degree_2}, Theorem \ref{prop: vanishing_cusp}, and the vanishing results for unitary spherical (Theorem \ref{thm: spherical}) and special representations (Theorem \ref{thm: special}) finish the proof of Theorem \ref{thm: A}. To conclude Theorem \ref{thm: B}, it remains to compute $\Hcb^2(G, \cH_\pi)$ when $\cM_\pi$ consists of centipedes.

Suppose that $S_0$ is a centipede or a spider of diameter $k$. 
Fix any two vertices $x,y\in S_0$ with degree $1$ such that $x,y$ belong to two different maximal complete proper subtrees of $S_0$. Define two subgroups $Q(x,y)$ and $\wt{Q}(x,y)$ of $\aut(S_0)$ as follows:
\begin{align*}
    Q(x,y)&:=\{\varphi\in \aut(S_0): \varphi(x)=x,\varphi(y)=y\},\\
    \wt{Q}(x,y)&:=\{\varphi\in \aut(S_0): \varphi(\{x,y\})=\{x,y\}\}.
\end{align*}
In other words, these are the point-wise and set-wise stabilizers of the set $\{x,y\}$ respectively. 


\begin{lem}\label{lem: centipede_alpha}
Let $k\geq 2$ and let $(\pi,\cH_\pi)$ be a unitary cuspidal representation of $G$ such that $\cM_\pi$ consists of centipedes or spiders of diameter $k$. Then there is a vector space isomorphism
$$
L_{\alt,\mathrm{w}^*}^\infty((\bd T)^2,\ell^\infty(\cM_\pi, V))^G\overset{\simeq}\longrightarrow V^{Q(x,y)}/\,V^{\wt{Q}(x,y)}$$
\begin{proof}

By Corollary \ref{cor: a.e. 2}, it is enough to find an isomorphism between the vector space $\cF_\alt^\infty((\bd T)^2, \ell^\infty(\cM_\pi, V))^G$ of bounded $G$-equivariant alternating maps and $V^{Q(x,y)}/\,V^{\wt{Q}(x,y)}$. 

Pick two distinct elements $\gamma_0,\gamma_1\in \bd T$ such that $S_0$ lies on the geodesic line $L:=L(\gamma_0,\gamma_1)$ with $\{x,y\}\subset L$. For a vertex $z$ of degree $q+1$ in $S_0$, assume without loss of generality that $x\in [z,\gamma_0)$ and $y\in [z,\gamma_1)$. Consider the linear map 
$$\ev:\,\cF_\alt^\infty((\bd T)^2, \ell^\infty(\cM_\pi, V))^G \longrightarrow V^{Q(x,y)}$$ 
defined by $\alpha \mapsto \alpha(\gamma_0,\gamma_1)(S_0)$. Observe that the image of $\ev$ is indeed contained in $V^{Q(x,y)}$: If $q\in Q(x,y)$, then by our choice of $\gamma_0$ and $\gamma_1$, there exists some $g\in \wt{G}(S_0)$ such that $p_{S_0}(g)=q$ with $g(\gamma_0,\gamma_1) = (\gamma_0, \gamma_1)$, and it follows that
$$\ev(\alpha)=\alpha(g\gamma_0,g\gamma_1)(gS_0)=\wt{\omega}(g)\alpha(\gamma_0,\gamma_1)(S_0)=\omega(q)\ev(\alpha)$$
for any $\alpha\in \cF_\alt^\infty((\bd T)^2, \ell^\infty(\cM_\pi, V))^G$. Now consider the linear map 
$$\wt{\ev}: \cF_\alt^\infty((\bd T)^2, \ell^\infty(\cM_\pi, V))^G\longrightarrow V^{Q(x,y)}/V^{\wt{Q}(x,y)}$$ 
defined by composing $\ev$ and the canonical projection $V^{Q(x,y)}\to V^{Q(x,y)}/V^{\wt{Q}(x,y)}$.

We first show that $\wt{\ev}$ is injective. Suppose that $\alpha(\gamma_0,\gamma_1)(S_0)\in V^{\wt{Q}(x,y)}$. Choose an element $\varphi\in V^{\wt{Q}(x,y)}$ such that $\varphi(x)=y$, $\varphi(y)=x$. Then by our choice of $\gamma_0$ and $\gamma_1$, there exists some $g\in \wt{G}(S_0)$ such that $p_{S_0}(g)=\varphi$ with $g\gamma_0=\gamma_1$ and $g\gamma_1=\gamma_0$. Then
$$\alpha(\gamma_0,\gamma_1)(S_0)=-\alpha(g\gamma_0,g\gamma_1)(gS_0) = -\omega(\varphi)\alpha(\gamma_0,\gamma_1)(S_0)=-\alpha(\gamma_0,\gamma_1)(S_0),$$
and so $\alpha(\gamma_0,\gamma_1)(S_0)=0$. By the $G$-equivariance of $\alpha$, it follows that $\alpha(\gamma_0,\gamma_1)(S)=0$ for all $S$ lying on $L=L(\gamma_0,\gamma_1)$. Moreover, if $S$ does not lie on $L$, then by Lemma \ref{lem: not_pass_head} we have $\alpha(\gamma_0,\gamma_1)(S)=0$, and so $\alpha(\gamma_0, \gamma_1)=0$. Finally, it follows from the transitivity of the action $G$ on $D_2(\bd T)$ that $\alpha=0$ on $D_2(\bd T)$, and hence on all of $(\bd T)^2$ since $\alpha$ is alternating.

We now show that $\wt{\ev}$ is surjective. Note that $Q(x,y)$ is a subgroup of index 2 in $\wt{Q}(x,y)$, so $Q(x,y)$ is normal in $\wt{Q}(x,y)$ and $V^{Q(x,y)}$ is invariant under the action of $\wt{Q}(x,y)$. Let $w\in V^{Q(x,y)}\setminus V^{\wt{Q}(x,y)}$. Then we may choose an element $\varphi\in \wt{Q}(x,y)$ of order 2 such that $\omega(\varphi)\ne \varphi$. We may then write $w=v_++v_-$, where $v_{\pm}$ is in the $(\pm1)$-eigenspaces of $\omega(\varphi)|_{V^{Q(x,y)}}$. Observe that $v_+\in V^{\wt{Q}(x,y)}$, and so both $w$ and $v_-$ define the same class in $V^{Q(x,y)}/V^{\wt{Q}(x,y)}$. Set $v:=v_-$. We define a map $\alpha: \bd T\times\bd T\to \ell^{\infty}(\cM_\pi, V)$ by
\begin{equation}\label{eq: construction_alpha}
\alpha(\gamma,\eta)(S):=\begin{cases}
    \wt{\omega}(s(S)^{-1}g)\cdot v & \textrm{if } $$(\gamma,\eta,S)=(g\gamma_0,g\gamma_1,gS_0)\in E_0$$\\
    0 & \textrm{otherwise,}
\end{cases}
\end{equation}
where $E_0$ denotes the $G$-orbit of the triple $(\gamma_0,\gamma_1,S_0)$ in $\bd T\times \bd T\times \cM_\pi$. The map $\alpha$ is well-defined. Indeed, if $(\gamma,\eta,
S)=g_1(\gamma_0,\gamma_1,S_0) = g_2(\gamma_0,\gamma_1,S_0)$, then
$$\wt{\omega}(s(S)^{-1}g_1)\cdot v=\wt{\omega}(s(S)^{-1}g_2)(\wt{\omega}(g_2^{-1}g_1)\cdot v).$$
Since $g_2^{-1}g_1$ fixes both $\gamma_0$ and $\gamma_1$, the image $p_{S_0}(g_2^{-1}g_1)$ in $\aut(S_0)$ is in $Q(x,y)$, and so $\wt{\omega}(g_2^{-1}g_1)\cdot v=v$. We now check that $\alpha$ is $G$-equivariant. If $(\gamma,\eta,S)=(g\gamma_0,g\gamma_1,gS_0)$, then for any $g'\in G$, one has
\begin{align*}
    \alpha(g'\gamma,g'\eta)(g'S)&= \wt{\omega}(s(g'S)^{-1}g'g)\cdot v\\
    &= \wt{\omega}(s(g'S)^{-1}g's(S))\wt{\omega}(s(S)^{-1}g)\cdot v\\
    &= \wt{\omega}(s(g'S)^{-1}g's(S))\alpha(\gamma,\eta)(S)
\end{align*}
as required. Finally, we check that $\alpha$ is alternating. By our choice of $\gamma_0,\gamma_1$ and $\varphi$, there exists a $g\in \wt{G}(S_0)$ such that $p_{S_0}(g)=\varphi$. It follows that
$$\alpha(\gamma_1,\gamma_0)(S_0) = \alpha(g\gamma_0,g\gamma_1)(gS_0)=\omega(\varphi)v =-v= -\alpha(\gamma_0,\gamma_1)(S_0).$$
It follows from the $G$-equivariance of $\alpha$ that $\alpha(\gamma,\eta)(S)=-\alpha(\eta,\gamma)(S)$ for all $(\gamma,\eta,S)\in \bd T\times \bd T\times \cM_\pi$, and so $\alpha$ is alternating. Thus, the map $\alpha$ is an element in $\cF_\alt^\infty((\bd T)^2,\ell^\infty(\cM_\pi, V))^G$ such that $\wt{\ev}(\alpha)=w$ as required.
\end{proof}
\end{lem}

\begin{rmk}\label{rmk: centipede_alpha}
\emph{The assumption of non-degeneracy was only there to ensure that for any $\alpha\in \cF_\alt^\infty ((\bd T)^2, \ell^\infty(\cM_\pi, V))^G$ and $(\gamma,\eta)\in (\bd T)^2$, the map $\alpha(\gamma,\eta)$ vanishes for all $S\in \cM_\pi$ not lying on $L(\gamma, \eta)$. }

\emph{Thus, if $\pi=\Ind_{S_0}\, \omega$ for any representation $(\omega,V)$ of $\aut(S_0)$, and $\cA_\omega$ is the space consisting of all $\alpha\in \cF_\alt^\infty ((\bd T)^2, \ell^\infty(\cM_\pi, V))^G$ such that for any $(\gamma,\eta)\in (\bd T)^2$, the map $\alpha(\gamma,\eta)$ vanishes for all $S\in \cM_\pi$ not lying on $L(\gamma, \eta)$, then the map}
$$\wt{\ev}: \cA_\omega\longrightarrow V^{Q(x,y)}/V^{\wt{Q}(x,y)}$$
\emph{considered in the proof of Lemma \ref{lem: a.e.zimmer} is still an isomorphism.} 
\end{rmk}


\begin{thm}\label{thm: H2_centipede}
Let $k\geq 2$ and let $(\pi,\cH_\pi)$ be a unitary cuspidal representation of $G$ such that $\cM_\pi$ consists of centipedes or spiders of diameter $k$. Then $$\Hcb^2(G, \cH_\pi)\simeq V^{Q(x,y)}/\,V^{\wt{Q}(x,y)}$$ 
as vector spaces.
\begin{proof}
Since the action of $G$ on $\bd T\times \bd T$ is doubly ergodic with coefficient in any continuous Banach $G$-modules, (\cite{Mon01}, Proposition 11.2.2), one has $L_{\alt}^\infty((\bd T)^2, \ell^2(\cM_\pi, V))^G=0$ and hence $\Hcb^2(G, \cH_\pi)\simeq \cZ L_{\alt,\mathrm{w}^*}^\infty((\bd T)^3,\ell^2(\cM_\pi, V))^G$ by Corollary \ref{cor: reso_bd}. By Lemma \ref{lem: finite_supp}, the cocycle $d\alpha$ defines a class in $\cZ L_{\alt,\mathrm{w}^*}^\infty((\bd T)^3,\ell^2(\cM_\pi, V))^G$, and it suffices to show that the boundary operator 
$$d:  L_{\alt,\mathrm{w}^*}^\infty((\bd T)^2,\ell^\infty(\cM_\pi, V))^G\to \cZ L_{\alt,\mathrm{w}^*}^\infty((\bd T)^3,\ell^2(\cM_\pi, V))^G$$
is an isomorphism between vector spaces by Lemma \ref{lem: centipede_alpha}. The surjectivity of $d$ follows again from the relative injectivity of $\ell^\infty(\cM_\pi, V)$ (Lemma \ref{lem: rel_inj}) and the fact that $\ell^2(\cM_\pi,V)\subset \ell^{\infty}(\cM_\pi, V)$ is a submodule, so it remains to show that $d$ is injective. To this end, we choose $\gamma_0, \gamma_1\in \bd T$ and define $\ev$ and $\wt{\ev}$ in the same way as in the first paragraph of the proof of Lemma \ref{lem: centipede_alpha}.

We first assume that $k\geq 3$. Let $H_x$, $H_y$ be the two heads containing $x$ and $y$ respectively, and denote by $S_x$ and $S_y$ the corresponding maximal complete proper subtree of $S_0$. Observe that the image $p_{S_0}(G(S_y))$ of $G(S_y)$ in $\aut(S_0)$ can be identified by $\aut(H_y)$, which is isomorphic to the permutation group of the set of vertices $S_0\setminus S_y$. For any two $a,b\in S_0\setminus  S_y$, we denote by $(ab)\in \aut(H_y)$ the transposition which permutes the two vertices. Now, let $\alpha\in \cF_\alt^\infty((\bd T)^2, \ell^\infty(\cM_\pi, V))^G$ be nonzero. Then the vector $v:= \ev(\alpha)=\alpha(\gamma_0,\gamma_1)(S_0)$ is a nonzero element in $V^{Q(x,y)}$ (see the first two paragraphs in the proof of Lemma \ref{lem: centipede_alpha}). Thus, for any $a,b\in S_0\setminus (S_y\cup \{y\})$, one has $\omega((ab))\cdot v=v$. Then there exists a vertex $y'\in S_0\setminus S_y$ such that $\omega((yy'))\cdot v\ne v$ by the non-degeneracy of $\omega$. Choose $\gamma_2\in \bd T$ such that the geodesic $L(\gamma_0,\gamma_2)$ passes through $y'$, and choose $g\in G(S_y)$ such that $f$ fixes $\gamma_0$ and swaps $\gamma_1$ and $\gamma_2$. Then $p_{S_0}(g)=(yy')$ and
$$
\alpha(\gamma_0,\gamma_2)(S_0) =\alpha(g\gamma_0,g\gamma_1)(gS_0) = \omega((yy'))\alpha(\gamma_0,\gamma_1)(S_0).$$
Moreover, we have by Lemma \ref{lem: not_pass_head} that $\alpha(\gamma_1,\gamma_2)(S_0)=0$ since $L(\gamma_1,\gamma_2)$ does not pass through the head $H_x$. Consequently, we have
$$
d\alpha(\gamma_0,\gamma_1,\gamma_2)(S_0) = -\alpha(\gamma_0,\gamma_2)(S_0)+\alpha(\gamma_0,\gamma_1)(S_0) = -\omega((yy'))\cdot v+v\ne 0,
$$
which shows that $d\alpha\ne 0$ as desired.

It remains to show that $d$ is injective when $S_0$ is a spider. Suppose that $d\alpha=0$, then by the relative injectivity of $\ell^\infty(\cM_\pi, V)$, it is enough to show that the space $L_{\mathrm{w}^*}^\infty(\bd T, \ell^\infty(\cM_\pi, V))^G$ of $0$-cochains is trivial. To this end, fix some $\gamma\in \bd T$ and let $P$ be the stabilizer of $\gamma$. We may then identify $\bd T$ with the quotient $G/P$. Since $G$ acts transitively on $\bd T$, there is a unique $G$-equivariant map $f_\beta$ for any $\beta\in L_{\mathrm{w}^*}^\infty(\bd T, \ell^\infty(\cM_\pi, V))^G$ (Remark \ref{rmk: a.e. orbits}), and we write $\beta(\eta):=f_\beta(\eta)$ for all $\eta\in \bd T$. Observe that $\beta(\gamma)\in \ell^\infty(\cM_\pi, V)$ is $P$-invariant for any $\beta$. Moreover, if $f\in \ell^\infty(\cM_\pi,V)$ is $P$-invariant, then
$$
\wt{\omega}(g)f(S_0) = \wt{\omega}(g)f(g^{-1}S_0) = (g\cdot f)(S_0) = f(S_0)
$$
for all $g\in P\cap \wt{G}(S_0)$. Thus, we have a well-defined composition of linear maps
\begin{align*}
   L_{\mathrm{w}^*}^\infty(\bd T, \ell^\infty(\cM_\pi, V))^G\overset{\psi_1}{\longrightarrow}\ell^\infty(\cM_\pi, V)^P\overset{\psi_2}\longrightarrow V^{P\cap \wt{G}(S_0)},
\end{align*}
where $\psi_1$ maps $\beta\in L_{\mathrm{w}^*}^\infty(\bd T, \ell^\infty(\cM_\pi, V))^G$ to $\beta(\gamma)$ and $\psi_2$ maps $f\in \ell^\infty(\cM_\pi, V)^P$ to $f(S_0)$. The map $\psi_1$ is injective since the action of $G$ on $\bd T$ is transitive. Since every $S\in \cM_\pi$ is completely determined by its unique vertex of valency $q+1$, the map $\psi_2$ is also injective as $P$ acts transitively on the set of vertices in $T$. Finally, let $x_0$ be the unique vertex in $S_0$ of valency $q+1$, and let $e_\gamma$ be the edge in $x_0$ that is contained in the chain $[x_0,\gamma)$. Then $e_\gamma$ is one of the maximal complete proper subtrees in $S_0$. Observe that the image of $P\cap \wt{G}(S_0)$ in $\aut(S_0)$ coincides with $p_{S_0}(G(e_\gamma))$. Thus, the vector space $V^{P\cap \wt{G}(S_0)}$ is trivial by the non-degeneracy of $\omega$, which completes the proof.
\end{proof}
\end{thm}

We provide an example to illustrate Theorem \ref{thm: H2_centipede}.
\begin{ex}
\emph{Let $S_0$ be a $k$-centipede in the $3$-regular tree $T_2$ ($k\geq 3$), then $S_0$ has two maximal proper complete subtrees $S_0'$ and $S_0''$. Let $x_1,x_2$ be the two vertices in $S_0\setminus S_0'$ and let $y_1,y_2$ be the two vertices in $S_0\setminus S_0''$. Note that $Q(x_1,y_1)$ is trivial and $\wt{Q}(x_1,y_1)$ contains only one non-trivial element $s$. Let $t\in \aut(S_0)$ be the element that flips $x_1$ and $x_2$ and fixes every other point in $S_0$. Then $s,t$ generates $\aut(S_0)$. In fact,
$$\aut(S_0)=\inprod{s,t\,| \, s^2=t^2=(st)^4=1}$$
is the dihedral group of order $8$. Let $(\omega, V)$ be the $1$-dimensional sign representation whose kernel is the subgroup generated by $st$. Then $\omega$ is non-degenerate and $V^{\wt{Q}(x_1,y_1)}$ is trivial since $\omega(s)=-1$. Let $\pi := \ind_{S_0}\omega$, then we have
$$\Hcb^2(G, \cH_\pi)\simeq V^{Q(x_1,y_1)}/\,V^{\wt{Q}(x_1,y_1)}\simeq V$$
by Theorem \ref{thm: H2_centipede}.
}
\end{ex}

The following observation, along with Theorem \ref{thm: H2_centipede} and \ref{prop: vanishing_cusp}, completes the proof of Theorem \ref{thm: B}:
\begin{prop}\label{prop: dim_1}
Let $S_0$ be a centipede or spider with diameter $k\geq 2$. Then
$$\dim\, V^{Q(x,y)}/V^{\wt{Q}(x,y)}\leq 1$$
for any irreducible representation $(\omega, V)$ of $\aut(S_0)$, where $x,y$ are chosen as before. In particular, one has
$$\dim \Hcb^2(G, \cH_\pi)\leq 1$$
for any unitary irreducible representation $\pi$ of $G$. 
\begin{proof}
We write $Q:=Q(x,y)$ and $\wt{Q}:=\wt{Q}(x,y)$ and fix an element $\wt{q}\in \wt{Q}\setminus Q$ of order 2. 

We start by assuming that $k\geq 3$. Denote by $H_x$ and $H_y$ the two heads in $S_0$ that contain $x$ and $y$ respectively, and denote by $S_x$, $S_y$ the two corresponding maximal complete proper subtrees of $S_0$. Consider the index-2 subgroup $\aut^+(S_0)$ in $\aut(S_0)$ defined by
$$\aut^+(S_0)=\{g\in \aut(S_0):\; gH_x=H_x\;\, \mathrm{and}\;\, gH_y=H_y\}.$$
We claim that $(\aut^+(S_0), Q)$ is a Gelfand pair. Indeed, suppose that $g\in \aut^+(S_0)$, then either $gx=g^{-1}x=x$, or both $gx$ and $g^{-1}x$ are vertices in $S\setminus (S_x\cup \{x\})$. Thus, one may choose $h_x\in \aut(H_x)\leq \aut^+(S_0)$ such that $h_xg^{-1}x=gx$ and $h_xx=x$. Similarly, we may choose $h_y\in \aut(H_y)$ such that $h_yg^{-1}y=gy$ and $h_yy=y$. Choose $q:= h_xh_y\in Q$, then $g^{-1}qg^{-1}$ stabilizes both $x$ and $y$ Thus, one has $g^{-1}\in QgQ$, which readily implies that $(\aut^+(S_0), Q)$ is a Gelfand pair (see \cite{Mac95}, VII.1). 

Now let $(\omega, V)$ be an irreducible representation of $\aut(S_0)$. Consider the restriction $\omega_0$ of $\omega$ to the subgroup $\aut^+(S_0)$. If $\omega_0$ is irreducible, then $\dim V^Q\leq 1$ since $(\aut^+(S_0), Q)$ is a Gelfand pair. Otherwise, we may write $V=W\oplus \wt{q}\,W$ as a direct sum of $\aut^+(S_0)$-modules for some irreducible representation $W$ of $\aut^+(S_0)$, and it follows that $\dim V^Q/V^{\wt{Q}}\leq 2$. If $V^Q\ne 0$, then either $W^Q\ne 0$ or $(\wt{q}\,W)^Q\ne 0$. Assume without loss of generality that $W^Q\ne 0$, then $w+\wt{q}w$ is a nontrivial element in $V^{\wt{Q}}$ for any $w\in V^Q\setminus \{0\}$, which implies that $\dim V^{\wt{Q}}\geq 1$. This finishes the proof for $k\geq 3$. 

Now suppose that $S_0$ is a spider. Then $\aut(S_0)\simeq  \Sym(q+1)$ and $\wt{Q}\simeq \Sym(q-1)\times \Z/2\Z$, and so $(\aut(S_0), \wt{Q})$ is a Gelfand pair (\cite{Sax81}). Let $(\omega, V)$ be an irreducible representation of $\aut(S_0)$. Then the tensor product $\omega\otimes \sgn$ between $\omega$ and the sign representation remains irreducible. Thus, the subspace $W$ of $\wt{Q}$-invariant vectors in $V$ with respect to $\omega\otimes \sgn$ has dimension at most 1. But $W$ coincides with the $(-1)$-eigenspace of $\omega(\wt{q})$ in $V^Q$, which is isomorphic to $V^Q/V^{\wt{Q}}$. 
\end{proof}
\end{prop}

\section{Explicit Constructions}\label{sec: C}

Let $S_0$ be a spider or a centipede. In this section, we will find or construct an explicit non-degenerate irreducible representation $\omega$ of $\aut(S_0)$ such that $\Hcb^2(\aut(T), \ind_{S_0}\,\omega)\ne 0$. We will prove Theorem \ref{thm: C} by showing that such $\omega$ is unique up to isomorphism. 

\subsection{The Spiders}\label{subsec: spider}

We start with the case where $S_0$ is a spider in $T_q$. Let $x_1,...,x_{q+1}$ be the vertices in $S_0$ of degree 1. For each $\sigma\in \Sym(q+1)$, we identify $\sigma$ with the element in $\Sigma$ that sends $v_j$ to $v_{\sigma(j)}$ for all $j=0,...,{q+1}$. Set
\begin{align*}
    Q&:=Q(x_{q},x_{q+1})\simeq \Sym(q-1),\\
    \wt{Q}&:=\wt{Q}(x_q,x_{q+1})\simeq \Sym(q-1)\times\Z/2\Z.
\end{align*}

Denote by $\{e_j\}$ the standard basis of $\C^{q+1}$.  Consider the representation $(\omega_\perm,V_\perm)$ of $\Sigma$, where $V_\perm=\C^{q+1}$ and $\Sigma$ acts on $V_\perm$ by permuting coordinates. Then $\omega_\perm$ is the direct sum of the trivial representation $\id$ and the \textbf{standard representation} $(\omega_\std,V_\std)$, where $V_\std$ corresponds to all vectors $v\in\C^{q+1}$ such that the sum of the coordinates of $v$ is zero. Note that both the trivial representation and the standard representation are not non-degenerate: the vector $-q\cdot e_{j+1}+\sum_{j=1}^q e_j\in V_\std$ is invariant under the action of $\stab_{\Sigma} (x_{q+1})$.

\begin{thm}\label{thm: wedge_spider}
Let $S_0$ be a spider in $T_q$. Then the second exterior power $\wedge^2\omega_{\std}$ of the standard representation is the unique non-degenerate irreducible representation $\omega$ of $\Sigma$ such that
$$\Hcb^2(\aut(T_q), \Ind_{S_0}\,\omega)\ne 0.$$
\end{thm}

To prove the theorem, we need to study a few specific irreducible representations of $\Sym(q+1)$. For background in representations of symmetric groups, see for instance, Section 5.12-5.17 in \cite{Eti+11}.

\begin{lem}\label{lem: std}
We have $\dim V_\std^Q=2$ and $\dim V_\std^{\wt{Q}}=1$.
\begin{proof}
It is evident that $\dim V_\std^{Q}=2$. In fact, the subspace $V_\std^Q$ is spanned by 
\begin{align*}
v_1&=(q-1)(e_q+e_{q+1})-\sum_{j=1}^{q=1}\,2\cdot e_j\;\;\;\textrm{and}\;\;\; v_2=-e_q+e_{q+1},
\end{align*}
of which $v_1\in V_{\std}^{\wt{Q}}$ and $v_2\notin V_\std^{\wt{Q}}$.
\end{proof}
\end{lem}

\begin{lem}\label{lem: (q-1,2)}
For $q\geq 3$, let $(\omega_{(q-1,2)},V_{(q-1,2)})$ be the irreducible representation of $\sym(q+1)$ that corresponds to the Young diagram below of the partition $(q-1,2)$:
\end{lem}
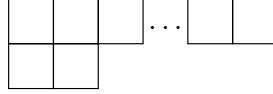
\ytableausetup{centertableaux}
\begin{center}
\begin{ytableau}
    \hfill & \hfill & \hfill & \none[\dots]
& \scriptstyle \hfill & \hfill \\
\hfill & \hfill  
\end{ytableau}
\end{center}
Then the restriction $\tau_q:=\Res_{\Sym(q+1)}^{Q} \,\omega_{(q-1,2)}$ contains the trivial representation $\id$ of $Q$. In fact, we have $\dim V_{(q-1,2)}^Q= 1$.

\begin{proof}
For $q=3$ or $4$, a direct comparison of the character of $\omega_{(q-1,2)}$ on the conjugacy classes of $\Sym(q-1)\leq \Sym(q+1)$ against the character table of representations of $\Sym(q-1)$ gives the result\footnote{The character tables could be found on \url{https://groupprops.subwiki.org/wiki/Category:Linear_representation_theory_of_particular_groups}}. Now we assume that $q\geq5$. Let $i=(i_1,...,i_{q-1})$ and let $C_i$ be the conjugacy class in $\Sym(q-1)$ with $i_j$ cycles of length $j$. Then by the Frobenius formula (see Section 5.15 in \cite{Eti+11}), the character $\chi_{\tau_q}(C_i)$ of $\tau_q$ is given by
\begin{align*}
    \chi_{\tau_q}(C_i) &= \left[(t_1-t_2)(t_1+t_2)^{i_1+2}\cdot\prod_{j=2}^{q-1}\;(t_1^j+t_2^j)^{i_j}\right]_{(q,2)}\\
    &= \left[t_1^2\cdot P(t_1,t_2)\right]_{(q,2)}+\left[2t_1t_2\cdot P(t_1,t_2)\right]_{(q,2)}+\left[t_2^2\cdot P(t_1,t_2)\right]_{(q,2)},
\end{align*}
where $[\,\cdot\,]_{j,k}$ denotes the coefficient of the term $t_1^jt_2^k$ in the polynomial inside the bracket, and $P(t_1,t_2)$ is defined by 
$$P(t_1,t_2):=(t_1-t_2)\cdot\prod_{j=1}^{q-1} (t_1^j+t_2^j)^{i_j}.$$
Since $\sum_{j=1}^{n-1}j\cdot i_j=q-1$, we immediately have $\left[t_2^2\cdot P(t_1,t_2)\right]_{(q,2)}=1$. Moreover, we compute that
\begin{align*}
    \chi_\std(C_i) &= \left[P(t_1,t_2)\right]_{(q-1,1)}=[t_1t_2\cdot P(t_1,t_2)]_{(q,2)},\\
    \chi_{(q-3,2)}(C_i)&= \left[P(t_1,t_2)\right]_{(q-2,2)}=\left[t_1^2\cdot P(t_1,t_2)\right]_{(q,2)},
\end{align*}
where $\chi_\std$ is the characters corresponding to the standard representation of $\Sym(q-1)$, and $\chi_{(q-3,2)}$ is the character corresponding to the irreducible representation of $\Sym(q-1)$ associated with partition $(q-3,2)$. It follows that 
$\chi_{\tau_q}=2\cdot\chi_\std+\chi_{(q-3,2)}+1$, and so $\tau_q = \omega_{\std}^{\oplus2}\oplus\omega_{(q-3,2)}\oplus \id$. In particular, we have $\dim V_{(q-1,2)}^Q=1$. 
\end{proof}

\begin{proof}[Proof of Theorem \ref{thm: wedge_spider}]
We first show that $\wedge^2\omega_\std$ is indeed a non-degenerate representation that gives non-trivial second bounded cohomology. To see that it is non-degenerate, consider the representation $\wedge^2\omega_\perm$. Observe that the subspace of $\wedge^2 V_{\perm}$ that is invariant under the action of $\stab_{\Sigma}(x_{q+1})\simeq \Sym(q)$ is spanned by the vector $v:=\sum_{j=1}^q e_j\wedge e_{q+1}$. But the projection of $v$ onto $\wedge^2V_\std$ is 0, which readily implies that $\wedge^2\omega_\std$ is non-degenerate. Now observe that the vector
$$\left(q\cdot e_q-e_{q+1}+\sum_{j=1}^{q-1} -e_j\right)\wedge\left(q\cdot e_{q+1}+\sum_{j=1}^q -e_j\right)$$
is in $(\wedge^2V_\std)^Q\setminus(\wedge^2V_\std)^{\wt{Q}}$, which shows that $\Hcb^2(\aut(T_q),\Ind_{S_0}\wedge^2\omega_\std)\ne 0$ by Theorem \ref{thm: H2_centipede}.

We now prove the uniqueness part of the Theorem. If $q=2$, then the only non-degenerate irreducible representation is the sign representation $\sgn$, which in this case coincides with $\wedge^2\omega_\std$. 

Now suppose that $q\geq 3$. Consider the regular representation $(\lambda, \C[\Sigma])$ of $\Sigma$ on the group ring $\C[\Sigma]$. Then the dimension of the $Q$-invariant subspace $\C[\Sigma]^Q$ is given by $|\Sigma|/|Q|$, which equals to $(q+1)q$. Recall that 
$$\lambda = \bigoplus_{(\omega, V_\omega)\in\widehat{\Sigma}} \omega^{\oplus \dim V_{\omega}},$$
where $\widehat{\Sigma}$ is the collection of all irreducible representations of $\Sigma$ 
. It follows that
\begin{align}\label{eq: dim_Q}
    \sum_{(\omega,V_\omega)\in \widehat{\Sigma}} \dim V_{\omega}\cdot\dim V_\omega^Q = (q+1)q.
\end{align}
Running the same argument for $\wt{Q}$, we also get
\begin{align}\label{eq: dim_Qt}
\sum_{(\omega,V_\omega)\in \widehat{\Sigma}} \dim V_{\omega}\cdot\dim V_\omega^{\wt{Q}} = \frac{(q+1)q}{2}.
\end{align}
By the hook length formula 
(see Section 5.17 in \cite{Eti+11}), one computes that $\dim V_{(q-1,2)}=(q+1)(q-2)/2$. Moreover, recall that $\dim V_\std=q$ and $\dim \wedge^2V_\std=q(q-1)/2$. Combining (\ref{eq: construction_alpha}) with Lemma \ref{lem: std}, Lemma \ref{lem: (q-1,2)} and the fact that $(\wedge^2 V_\std)^Q\ne 0$, we conclude that the $\omega_\std$, $\wedge^2 \omega_{\std}$, $\omega_{(q-1,2)}$, and the trivial representation $\id$ are all irreducible representations of $\Sigma$ with nonzero $Q$-invariant vectors. Since $\wt{Q}\leq Q$, these are also all irreducible representations that could possibly have $\wt{Q}$-invariant vectors. Thus, combining Lemma \ref{lem: std}, Lemma \ref{lem: (q-1,2)} and (\ref{eq: dim_Qt}), one gets
$$\dim V_{(q-1,2)}^{\wt{Q}}=\dim V_{(q-1,2)}^Q=1.$$
Since $\omega_\std$ is not non-degenerate, we have by Theorem \ref{thm: H2_centipede} that $\wedge^2\omega_\std$ is the only irreducible non-degenerate representation $\omega$ such that $\Hcb^2(\aut(T_q), \Ind_{S_0}\,\omega)\ne0$: all other irreducible representations $(\omega, V_\omega)$ of $\Sigma$ are degenerate or satisfy $V_\omega^Q/V_{\omega}^{\wt{Q}}=0$. 
\end{proof}

\subsection{The Centipedes}\label{subsec: cent}

Now let $S_0$ be a $k$-centipede ($k\geq 3$) in the $(q+1)$-regular tree $T_q$. Then $S_0$ has two heads $H_1$ and $H_2$. Label the degree-1 vertices in $H_1$ and $H_2$ by $x_1,...,x_{q}$ and $y_1,...,y_{q}$, respectively. Let $\Sigma:= \aut(S_0)$. We identify the fixator subgroups $\mathrm{Fix}_\Sigma(S_0\setminus\{x_1,...,x_q\})$ and $\mathrm{Fix}_\Sigma(S_0\setminus\{y_1,...,y_q\})$ with $\aut(H_1)$ and $\aut(H_2)$ respectively, and further identify $\aut(H_j)$ with $\Sym(q)$ as in the first paragraph of the previous section. 

Set $Q:=Q(x_q,y_q)$ and $\wt{Q}:=\wt{Q}(x_q,y_q)$. Fix once for all an order-2 element $\wt{q}\in \wt{Q}\setminus Q$ that flips $x_j$ and $y_j$ for $j=1,...,q$. If $\Sigma^+:=\aut^+(S_0)$ is the index-2 subgroup of $\Sigma$ defined in the proof of proposition \ref{prop: dim_1} and $(\omega,V_\omega)$ is a representation of $\Sigma^+$, then we may consider its conjugate representation $(\omega^{\wt{q}},V_\omega)$ of $\Sigma^+$ defined by
$\omega^{\wt{{q}}}(g):=\omega(\wt{q}^{-1}g\wt{q}).$
If there is an isomorphism $\Phi:V_\omega\to V_\omega$ between the two representations $\omega$ and $\omega^{\wt{q}}$, then one can extend $(\omega,V_\omega)$ to a representation $(\wt\omega,V_\omega)$ of the whole group $\Sigma$ by defining $\wt{\omega}(\wt q):= \Phi$. 

{Observe that we can identify $\Sigma^+$ with the direct product
$$\Sym(q)\times \Sym(q-1)^{k-3}\times\Sym(q).$$
Thus, if $(\omega_1,V_1)$ and $(\omega_2,V_2)$ are two irreducible representations of $\Sym(q)$, then the representation
$$(\omega_1\otimes\id^{\otimes{k-3}}\otimes\omega_2,\;V_1\otimes \C^{\otimes{k-3}}\otimes V_2),$$
which we abusively abbreviate as $(\omega_1\otimes\omega_2, V_1\otimes V_2)$, is an irreducible representation of $\Sigma^{+}$. 

Now consider the standard representation $(\omega_\std,V_\std)$ of $\Sym(q)$, then $\omega_\std\otimes\omega_\std$ is isomorphic to its conjugate representation $(\omega_\std\otimes\omega_\std)^{\wt{q}}$ via the map $\Phi_\std:V_\std\otimes V_\std\to V_\std\otimes V_\std$ defined by $\Phi_\std(v\otimes w)=w\otimes v$. Thus, it can be extended to an irreducible representation $(\hat{\omega}_\std, V_\std^{\otimes 2})$ of $\Sigma$ of degree $(q-1)^2$}.

Denote by $\varepsilon$ the 1-dimensional representation of $\Sigma$ whose kernel is $\Sigma^+$.

\begin{lem}\label{lem: centipedes_std}
Both $\hat{\omega}_\std$ and the tensor product $\hat{\omega}'_\std:=\hat{\omega}_\std\otimes\varepsilon$ are non-degenerate representations of $\Sigma$ with a nonzero $Q$-invariant vector. Moreover, one has
$$\Hcb^2(\aut(T_q), V_{\hat{\omega}_\std})=0 \;\;\,\textrm{and}\;\;\,\Hcb^2(\aut(T_q), V_{\hat{\omega}_\std'})\ne 0.$$
\begin{proof}
For $j=1,2$, the restriction of $\omega_\std\otimes\omega_\std$ to the subgroup $\aut(H_j)$ is isomorphic to the standard representation of $\Sym(q)$, which does not contain the trivial representation. This readily implies the non-degeneracy of both $\hat{\omega}_\std$ and $\wt{\omega}'_\std$. Now, observe that the vector
$$v:=\left(-(q-1)\cdot e_q+\sum_{j=1}^{q-1}e_j\right)\otimes \left(-(q-1)\cdot e_q+\sum_{j=1}^{q-1}e_j\right)\in V_\std\otimes V_\std$$
is a non-zero $Q$-invariant vector. Since $(\Sigma^+,Q)$ is a Gelfand pair (see Proposition \ref{prop: dim_1}), the subspace $(V_\std\otimes V_\std)^Q$ is spanned by $v$. Since $v$ is also invariant under $\Phi_\std=\hat{\omega}_\std(\wt{q})$, one concludes that $v\in V_{\wt\omega_\std}^{\wt{Q}}$ and it follows that 
$$\Hcb^2(\aut(T_q), V_{\hat{\omega}_\std}^{\wt{Q}})=0$$
by Theorem \ref{thm: H2_centipede}. Finally, since $\hat{\omega}_\std'(\wt{q})v=-v$, one has $v\notin V_{\hat{\omega}_\std'}^{\wt{Q}}$ and so
$$\Hcb^2(\aut(T_q), V_{\hat{\omega}_\std'})\ne 0$$
by Theorem \ref{thm: H2_centipede}.
\end{proof}
\end{lem}

\begin{thm}\label{thm: centipede_rep}
Denote by $\varepsilon$ the 1-dimensional representation of $\Sigma$ whose kernel is $\Sigma^+$. Then the tensor product $\hat{\omega}_\std':=\hat{\omega}_\std\otimes\varepsilon$ is the unique non-degenerate irreducible representation $\omega$ of $\Sigma$ such that
$$\Hcb^2(\aut(T_q), \Ind_{S_0} \, \omega)\ne0.$$
\begin{proof}
We follow the same strategy used in the proof of Theorem \ref{thm: wedge_spider}: Consider the regular representation $(\lambda,\C[\Sigma])$, then
$$\dim \C[\Sigma]^Q = \frac{|\Sigma|}{|Q|} = \frac{2|\Sigma^+|}{|Q|}=\frac{2\cdot|\Sym(q)\times\Sym(q-1)^{k-3}\times\Sym(q)|}{|\Sym(q-1)^{k-1}|}=2q^2.$$
It follows that 
\begin{align}\label{eq: 2q^2}
\sum_{(\hat\omega,V_{\hat\omega})\in \widehat{\Sigma}} \dim V_{\hat\omega}\cdot\dim V_{\hat\omega}^Q = 2q^2.
\end{align}

Now consider the representation $(\id\otimes \omega_\std, \C\otimes V_\std)$ of $\Sigma^+$. Then note that it is inequivalent to its conjugate representation $(\id\otimes\omega_\std)^{\wt{q}}$, which is isomorphic to $(\omega_\std\otimes\id, V_\std\otimes\C)$. Thus, the induced representation of $\Sigma$
$$\hat{\omega}_{2(q-1)}:=\Ind_{\Sigma^+}^\Sigma \,(\id\otimes\omega_\std)$$
is irreducible (see \cite{Ser77}, Section 7.4). Moreover, both $\id\otimes \omega_\std$ and $\omega_\std\otimes\id$ admit nontrivial $Q$-invariant vectors: if $v :=-(q-1)\cdot e_q+\sum_{j=1}^{q-1}e_j$, then $1\otimes v$ and $v\otimes 1$ are $Q$-invariant vectors for $\id\otimes \omega_\std$ and $\omega_\std\otimes\id$ respectively. Since  
$$\Res_\Sigma^{\Sigma^+}\hat{\omega}_{2(q-1)}\simeq (\id\otimes \omega_\std)\oplus (\id\otimes \omega_\std)^{\wt{q}}, $$
we have that $\dim V_{\hat{\omega}_{2(q-1)}}^Q=2$, and so\, $(\dim V_{\hat{\omega}_{2(q-1)}})\cdot (\dim V_{\hat{\omega}_{2(q-1)}}^Q)=4(q-1)$. 

Finally, by Lemma \ref{lem: centipedes_std} and (\ref{eq: 2q^2}), we conclude that the only irreducible representations of $\Sigma$ that could admit nontrivial $Q$-invariant vectors are the trivial representation, $\varepsilon$, $\hat{\omega}_\std$, $\hat{\omega}_\std'$, and $\hat{\omega}_{2(q-1)}$. Since $\varepsilon$ and $\hat{\omega}_{2(q-1)}$ are both degenerate, the proof is complete by Lemma \ref{lem: centipedes_std} and Theorem \ref{thm: H2_centipede}.
\end{proof}
\end{thm}

Theorem \ref{thm: C} now follows from Theorem \ref{thm: wedge_spider} and \ref{thm: centipede_rep}.

\section{The Cocycle of Monod and Shalom}\label{sec: D}

In this section, we recall the construction of the cocycle by Monod and Shalom in \cite{MS04} (see also \cite{MS03}) and explain its relationship to our previous computations. We will address Theorem \ref{thm: D} in the introduction.

Let $\cE_{T}$ be the collection of edges in a tree $T$, and let $\vee^k\cE$ be the collection of all $k$-tuples $(\xi_1,...,\xi_k)\in \cE_T^k$ such that $\xi_1,...,\xi_k$ forms a finite sequence of non-backtracking consecutive edges in $T$. In this case, we write $\mathbf{\e}=\xi_1\vee...\vee \xi_k$ and denote by $\mathbf{e}^{-1}$ the reversed sequence $\xi_k\vee...\vee \xi_1$. 

Now let $G=\aut(T_q)$ and write $\cE:=\cE_{T_q}$. Consider the $G$-module
$(\lambda,\ell^p_\alt(\vee^k\cE))$ consisting of all functions $f\in\ell^p(\vee^k\cE_T)$ with
$$f(\mathbf{e})=-f(\mathbf{e}^{-1})$$
for all $\mathbf{e}\in \vee^k\cE_T$, where $\aut(T)$ acts by left translations. For each $k\geq 2$, a nontrivial class in $\Hcb^2(G, \ell_\alt^2(\vee^k\cE))$ was constructed by Monod and Shalom in \cite{MS04} as follows: for any $\gamma,\eta\in\bd T$, define $\alpha^k_{\textrm{ms}}(\gamma,\eta)\in \ell_\alt^\infty(\vee^k\cE)$ by
$$\alpha^k_{\textrm{ms}}(\gamma,\eta)(\mathbf{e}):=
\begin{cases}
    1 & \textrm{if $(\gamma,\mathbf{e},\eta)$ is positively oriented}\\
    -1 & \textrm{if $(\gamma,\mathbf{e},\eta)$ is negatively oriented}\\
    0 & \textrm{otherwise.}
\end{cases}
$$
Here, we say that $(\gamma,\mathbf{e}=(e_1,...,e_k), \eta)$ is \textbf{positively oriented} if $\mathbf{e}$ lies on the bi-infinite geodesic $L(\gamma,\eta)$ joining $\gamma$ and $\eta$ for all $j$ and $e_j$ is closer to $\gamma$ compared to $e_{j-1}$ for all $j=2,...,k$. We say that $(\gamma,\mathbf{e}, \eta)$ is \textbf{negatively oriented} if $(\gamma,\mathbf{e}^{-1}, \eta)$ is positively oriented. Then the coboundary
$$c^k_{\textrm{ms}}:=d\alpha^k_{\textrm{ms}}: \bd T\times\bd T\times\bd T\longrightarrow \ell^2(\vee^k\cE)$$
is a nontrivial cochain in degree 2. From now on, we fix $k \geq 2$ and write $c_\ms$ and $\alpha_\ms$ instead of $c_\ms^k$ and $\alpha_\ms^k$.

In this subsection, we will decompose the module $\ell_\alt^2(\vee^k\cE)$ and describe $\alpha_\textrm{ms}$ in terms of the nontrivial cocycle that corresponds to Theorem \ref{thm: wedge_spider} and \ref{thm: centipede_rep}. To this end, we first give an alternative description of $\ell_\alt^2(\vee^k\cE)$ in terms of centipedes or spiders.

Fix a $k$-centipede (or a spider if $k=2$) $S_0$ in $T_q$. Let $\Sigma = \aut(S_0)$ and write $\cE_0:=\cE_{S_0}$. Consider the $\Sigma$-module $(\omega_{\textrm{ms}},V_{\mathrm{ms}})$, where $V_\ms$ consists of all alternating functions $f:\vee^k\cE_0\longrightarrow\C$ on which $\Sigma$ acts by left translations. Note that $\vee^k\cE_0$ can be identified with the set of all ordered pairs $(x,y)\in S_0^2$ of two vertices $x$ and $y$ such that $d(x,y)=k$, and we will write $\e \in \vee^{k}\cE_0$ as $x\vee y$ if $x$ and $y$ are the two endpoints of $\e$. 

\begin{prop}\label{prop: ms_centipede_rep}
For $p\in [1,+\infty]$, equip the finite-dimensional vector space $V_\ms$ with the $\ell^p$-norm and let $\pi_{\ms,p}$ be the $L^p$-induction $L^p\Ind_{\wt{G}(S)}^G \wt{\omega}_{\ms}$. Then $\pi_{\ms,p}$ is isometrically isomorphic to $(\lambda,\ell_\alt^p(\vee^k \cE))$ for $1\leq p\leq +\infty$. If $p = 2$, then the isomorphism is a unitary equivalence.
\begin{proof}
Write $\pi :=\pi_{\ms,p}$. As usual, we fix a set theoretic section $s:\cM_\pi\to G$ such that $s(S_0)=e_G$. Then 
$$(\pi(g)f)(S) = \wt{\omega}_\ms(s(S)^{-1}gs(g^{-1}S))f(g^{-1}S)$$
for all $g\in G, f\in \ell^p(\cM_\pi, V_\ms)$ and $S\in\cM_\pi$.

Given a map $f\in \ell^p(\cM_\pi, V_\ms)$, we define $\varphi f:\vee^k\cE\rightarrow\C$ as follows: for any $\mathbf{e}:=e_1\vee...\vee e_k$, there is a unique $S_\e\in \cM_\pi$ that contains $\mathbf{e}$. Define
$$\varphi f(\mathbf{e}):=f(S_\e)(s(S_\e)^{-1}\e).$$
Then $\varphi f$ defines an element in $\ell_\alt^p(\vee^k\cE)$. Conversely, given $f\in \ell_\alt^p(\vee^k\cE)$, one may define $\psi f: \cM_\pi\to V_\ms$ by
$$((\psi f)(S))(\e):=f(s(S)\e)\quad\forall S\in \cM_\pi, \,\e\in \vee^k\cE_0.$$
Then $\psi: \ell_\alt^p(\vee^k\cE)\to \ell^p(\cM_\pi,V_\ms)$ is the inverse of $\varphi: \ell^p(\cM_\pi,V_\ms)\rightarrow \ell_\alt^p(\vee^k\cE)$.
Indeed, we have
$$(\varphi(\psi f))(\e)=((\psi f)(S_\e))(s(S_\e)^{-1}\e)=f(\e)$$
for all $f\in\ell_\alt^p(\vee^k\cE)$ and $\e\in \vee^k\cE$. Conversely, we have for any $f\in \ell^p(\cM_\pi, V_\ms)$, $S\in \cM_\pi$ and $\e\in \vee^k\cE_0$,
$$((\psi(\varphi f))(S))(\e)= (\varphi f)(s(S)\e)=f(S)(\e),$$
where the last equality follows from the observation that $S_{s(S)\e} = S$. This shows in particular that $\varphi$ is bijective. To see that $\varphi$ is $G$-equivariant, we compute that
\begin{align*}
    (\lambda(g)(\varphi f))(\e) = \varphi f(g^{-1}\e) = f(g^{-1}S_{\e})(s(g^{-1}S_\e)^{-1}(g^{-1}\e))
\end{align*}
for $g\in G$, $f\in \ell_\alt^p(\cM_\pi,V_\ms)$ and $\e\in\vee^k\cE$. On the other hand, we have
\begin{align*}
    ((\varphi\circ\pi(g))f)(\e) &= ((\pi(g)f)(S_\e))(s(S_\e)^{-1}\e)\\
    &=(\wt{\omega}_\ms(s(S_\e)^{-1}gs(g^{-1}S_\e))\cdot f(g^{-1}S_\e))(s(S_\e)^{-1}\e)\\
    &=f(g^{-1}S_{\e})(s(g^{-1}S_\e)^{-1}(g^{-1}\e)),
\end{align*}
which shows that $\varphi$ is $G$-equivariant. If $p=+\infty$, then it is clear that both $\varphi$ and $\psi$ are of norm at most one, which shows that $\varphi$ is an isometry. Now suppose that $p=2$, then $\varphi$ is unitary since
\begin{align*}
\inprod{\varphi f_1,\varphi f_2} &= \sum_{\e\in \vee^k\cE} f_1(S_\e)(s(S_e)^{-1}\e)\cdot \overline{f_2(S_\e)(s(S_e)^{-1}\e)}\\
&= \sum_{S\in \cM_\pi}\,\sum_{\e\in \vee^k\cE_{S}} f_1(S)(s(S)^{-1}\e)\cdot \overline{f_2(S)(s(S)^{-1}\e)}\\
&= \sum_{S\in \cM_\pi}\,\sum_{\e\in \vee^k \cE_0} f_1(S)(\e)\cdot\overline{f_2(S)(\e)} = \inprod{f_1,f_2}
\end{align*}
for all $f_1,f_2\in \ell^2(\cM_\pi,V_{\ms})$. A similar computation shows that $\varphi$ is an isometry for all $p<+\infty$. 
\end{proof}
\end{prop}

We take a closer look at the module $V_\ms$. If $S_0$ is a spider, then we label the $(q+1)$ vertices of degree $1$ by $x_1,...,x_{q+1}$. We identify $\Sigma\simeq \Sym(q+1)$ and define $Q = Q(x_{q},x_{q+1})$ and $\wt{Q}$ as in Subsection \ref{subsec: spider}. For future use, we let $\wt{q}\in \wt{Q}\setminus Q$ be the transposition that flips $x_{q}$ and $x_{q+1}$.

\begin{lem}\label{lem: spider_ms}
Suppose that $S_0$ is a spider in $T_q$. Let $(\omega_\perm, V_\perm)$ be the permutation representation defined in the beginning of Section \ref{subsec: spider}. The second exterior product $\wedge^2 \omega_\perm$ is then isomorphic to $\omega_\ms$. 
\begin{proof}
Define $f_{ij}\in V_\ms$ by
$$
f_{ij} (x_k\vee x_l) = 
\begin{cases}
    1 & \textrm{if $(k,l)=(i,j)$}\\
    -1 & \textrm{if $(k,l)=(j,i)$}\\
    0 & \textrm{otherwise}
\end{cases}
$$
for all $i\ne j$. Then $\{f_{ij}\}_{i<j}$ forms a basis of $V_\ms$. Consider the map $\varphi:V_\ms\to \wedge^2V_\perm$ defined by sending $f_{ij}$ to $e_i\wedge e_j$. It is a direct verification that $\varphi$ is $\Sigma$-equivariant. Moreover, it is an isomorphism since it sends a basis of $V_\ms$ to a basis of $\wedge^2 V_\perm$. 
\end{proof}
\end{lem}

Now suppose that $S_0$ is a $k$-centipede for $k\geq 3$ and let $H_1$, $H_2$ be the two heads of $S_0$. Label the degree-1 vertices in $H_1$ and $H_2$ by $x_1,...,x_q$ and $y_1,...,y_q$ respectively. We identify $\aut(H_j)$ with $\Sym(q)$ and define $Q:=Q(x_q,y_q)$ and $\wt{Q}:=\wt{Q}(x_q,y_q)$ as in subsection \ref{subsec: cent}. Fix an order-2 element $\wt{q}\in \Sigma$ that sends $x_j$ to $y_j$  Consider the permutation representation $\omega_\perm$ of $\Sym(q)$. Then $\omega_\perm\otimes\omega_\perm$ defines a representation of the index-2 subgroup $\Sigma^+$ that is isomorphic to its conjugate representation $(\omega_\perm\otimes\omega_\perm)^{\wt{q}}$ through the map $\Phi':v\otimes w\mapsto -w\otimes v$. Thus, it can be extended to a representation $(\hat{\omega}_\perm', V_\perm^{\otimes 2})$ by defining $\hat{\omega}'_\perm(\wt{q})=\Phi'$.

\begin{lem}\label{lem: centipede_ms}
Let $S_0$ be a $k$-centipede. Then $\hat{\omega}_\perm'\simeq \omega_\ms$.
\begin{proof}
Similarly to Lemma \ref{lem: spider_ms}, define $f_{ij}$ by
$$
f_{ij} (x_k\vee y_l) = 
\begin{cases}
    1 & \textrm{if $(k,l)=(i,j)$}\\
    0 & \textrm{otherwise}
\end{cases}
$$
for all $i\ne j$. Then $\{f_{ij}\}_{i,j\in{1,...,q}}$ is a basis of $V_\ms$. Now, the map $V_\ms\to V_\perm\otimes V_\perm$ defined by sending $f_{ij}$ to $e_i\otimes e_j$ is an isomorphism between $\Sigma$-modules. 
\end{proof}
\end{lem}

We have thus proved the first part of Theorem \ref{thm: D}:
\begin{cor}\label{cor: ms_centipede_rep}
If $\hat\omega$ is the unique irreducible nondegenerate representation described in Theorem \ref{thm: wedge_spider} and Theorem \ref{thm: centipede_rep} such that $\Hcb^2(G, \Ind_{S_0}\,\hat\omega)\ne 0$, then $\hat\pi:=\Ind_{S_0}\,\hat\omega$ is a subrepresentation of $(\lambda,\ell_\alt^2(\vee^k \cE))$. 
\begin{proof}
Since $\omega_\std$ is a subrepresentation of $\omega_\perm$, the statement follows directly from Proposition \ref{prop: ms_centipede_rep} along with Lemma \ref{lem: spider_ms} and \ref{lem: centipede_ms}.
\end{proof}
\end{cor}

We now study the relationship between $c_\ms = d\alpha_\ms$ and the nontrivial classes in the 1-dimensional $\Hcb^2(G, \Ind_{S_0}\,\hat\omega)$ and undertake the proof of the second part of Theorem \ref{thm: D}

To this end, choose $\gamma_0,\gamma_1\in \bd T$ such that $S_0$ lies on $L(\gamma_0, \gamma_1)$. 
If $S_0$ is a spider, we assume that $x_q,x_{q+1}\in L(\gamma_0,\gamma_1)$ and that $(\gamma_0, x_q\vee x_{q+1},\gamma_1)$ is positively oriented. If $S_0$ is a centipede, then we similarly assume that $x_q,y_q\in L(\gamma_0,\gamma_1)$ and that $(\gamma_0, x_q\vee y_{q},\gamma_1)$ is positively oriented. 

Let $\omega$ be any representation of $\aut(S_0)$, and let $\pi:=\Ind_{S_0}\, \omega$. Denote by $\cA_\omega$ the space consisting of all $\alpha\in \cF_\alt ((\bd T)^2, \ell^\infty(\cM_\pi, V))^G$ such that for any $(\gamma,\eta)\in (\bd T)^2$, the map $\alpha(\gamma,\eta)$ vanishes for all $S\in \cM_\pi$ not lying on $L(\gamma, \eta)$. Then by Remark \ref{rmk: centipede_alpha}, the map
$$\wt{\ev}: \cA_\omega\longrightarrow E_-(\omega,\wt{q})$$
defined by $\wt{\ev}(\alpha):=\alpha(\gamma_0,\gamma_1)(S_0)$
is an isomorphism. Here, $E_-(\omega,\wt{q})\simeq V^{Q}/V^{\wt{Q}}$ is the $(-1)$-eigenspace of $\omega(\wt{q})$ in $V^Q$. If $v\in E_-(\omega,\wt{q})$, then recall that its preimage $\alpha:=\wt{\ev}^{-1}(v)$ is given by (\ref{eq: construction_alpha}) in Lemma \ref{lem: centipede_alpha}.

Since $\hat\omega$ is a subrepresentation of $(\omega_\ms, V_\ms)$, we may 
consider the orthogonal projection $\hat P$ of $V_\ms$ onto $V_{\hat\omega}$ (with respect to the $\ell^2$-inner product). This induces a map 
$$\hat P_*: \cF_\alt^\infty ((\bd T)^{n+1}, \ell^p(\cM_{\pi_\ms}, V_\ms))^G\longrightarrow \cF_\alt^\infty ((\bd T)^{n+1}, \ell^p(\cM_{\hat\pi}, V_{\hat\pi}))^G$$
for $p\in [1,+\infty]$ where $\pi_\ms = \Ind_{S_0}\,\omega_\ms$ and $\hat\pi = \Ind_{S_0}\,\hat\pi$. We restate below the second part of Theorem \ref{thm: D} under the above notation:

\begin{thm}\label{Thm: MS}
Let $\kappa:= [\hat P_*c_\ms]\in \Hcb^2(G, \Ind_{S_0}\,\hat\omega)$. Then $\Hcb^2(G, \Ind_{S_0}\,\hat\omega) \simeq \C\cdot\kappa.$
\end{thm}

This will be a direct consequence of the following observation:
\begin{lem}\label{lem: commute}
Let $\omega$ be any representation of $\Sigma$ and let $\pi = \Ind_{S_0}\,\omega$. Suppose that $\omega=\omega_1\oplus\omega_2$ is a direct sum of its two subrepresentations $(\omega_1,V_1)$, $(\omega_2, V_2)$, then the diagram
\begin{equation*}\label{eq: comparison}
\begin{tikzcd}
    \cA_\omega \arrow[r, "\wt{\ev}"] \arrow[d, "(P_j)_*"'] & E_-(\omega,\wt{q})\arrow[d, "P_j"] \\
    \cA_{\omega_j} \arrow[r, "\wt{\ev}"] & E_-(\omega_j,\wt{q})
\end{tikzcd}
\end{equation*}
is commutative for each $j=1,2$. Here, $P_j$ is the projection of $V_\omega$ onto $V_j$, which induces a map $(P_j)_*:\cA_\omega\to\cA_{\omega_j}$. 
\begin{proof}
Let $\alpha\in\cA_\omega$. Then $\alpha=\wt{\ev}^{-1}(v)$ for $v = \alpha(\gamma_0,\gamma_1)(S_0)$. It is enough to show that 
$$(P_j)_*\alpha=\wt{\ev}^{-1}(P_jv).$$
To this end, we compute that
\begin{align*}
((P_j)_*\alpha)(g\gamma_0,g\gamma_1)(gS_0) &= P_j(\alpha(g\gamma_0,g\gamma_1)(gS_0))  
= P_j(\wt{\omega}(s(gS_0)^{-1}g)\cdot v)\\
&=\wt{\omega}(s(gS_0)^{-1}g)\cdot (P_jv).
\end{align*}
Moreover, if $(\gamma,\eta,S)\in \bd T\times\bd T\times\cM_\pi$ is not in the $G$-orbit of $(\gamma_0,\gamma_1,S_0)$, then $S$ cannot lie on $L(\gamma,\eta)$ and so $\alpha(\gamma,\eta)(S)=(P_j)_*\alpha=0$. It follows that $(P_j)_*\alpha$ agrees with the definition of $\wt{\ev}^{-1}(P_jv)$.  
\end{proof}
\end{lem}

\begin{proof}[Proof of Theorem \ref{Thm: MS}]
Note that in the Theorem, we identified $(\lambda, \ell_\alt^p(\vee^k \cE))$ with $L^p\Ind_{\wt{G}}^G\,\wt{\omega}_\ms$ via Proposition \ref{prop: ms_centipede_rep}.

Since $\Hcb^2(G,\Ind_{S_0}\,\hat{\omega})$ is 1-dimensional by Proposition \ref{prop: dim_1}, it is enough to show that $\kappa\ne 0$, which amounts to showing that $\hat{P}_*\alpha_\ms\ne 0$ by the proof of Theorem \ref{thm: H2_centipede}. Set $v := \alpha_\ms(\gamma_0,\gamma_1)(S_0)\in V_\ms$, then it suffices to show that $\hat Pv\ne 0$ and
$$\hat P_*\alpha_\ms = \wt{\ev}^{-1}(\hat{P}v).$$
To this end, consider $\alpha_\ms$ as a map 
$$\alpha_\ms:\bd T\times \bd T\longrightarrow \ell^\infty(\cM_{\pi_\ms}, V_\ms)$$
by Proposition \ref{prop: ms_centipede_rep}, then $\alpha_\ms \in \cA_{\omega_\ms}$ and consequently $\alpha_\ms = \wt{\ev}^{-1}(v)$. In fact, we have
$$
v = \begin{cases}
    e_q\wedge e_{q+1}& \textrm{if $S_0$ is a spider}\\
    e_q \otimes e_q & \textrm{if $S_0$ is a $k$-centipede for $k\geq 3$,}
\end{cases}
$$
where we identified $V_\ms$ with $\wedge^2\omega_\perm$ and $\hat{\omega}_\perm$ respectively by Lemma \ref{lem: spider_ms} and \ref{lem: centipede_ms}. It is now evident that $\hat{P}v\ne0$. The fact that $\hat P_*\alpha_\ms = \wt{\ev}^{-1}(\hat{P}v)$ follows directly from Lemma \ref{lem: commute}.
\end{proof}
  


\printbibliography

\end{document}